\documentclass[12pt]{article}
\usepackage{a4,amsmath,amssymb, amsthm, latexsym, color, graphicx}
\usepackage{tikz}
\usetikzlibrary{arrows.meta}
\usetikzlibrary{decorations.markings}
\usetikzlibrary{calc}
\usepackage{diagbox}

\newtheorem{theorem}{Theorem}[section]
\newtheorem{proposition}[theorem]{Proposition}
\newtheorem{lemma}[theorem]{Lemma}
\newtheorem{corollary}[theorem]{Corollary}

\theoremstyle{definition}
\newtheorem{example}[theorem]{Example}
\newtheorem{definition}[theorem]{Definition}

\title{New results on non-disjoint and classical strong external difference families}
\author{Sophie Huczynska\thanks{Email: sh70@st-andrews.ac.uk (Sophie Huczynska)}  \,and Sophie Hume\thanks{Email: 5sophie5555@gmail.com (Sophie Hume)}}
\date{School of Mathematics and Statistics, University of St Andrews,\\
St Andrews, Fife, KY16 9SS,UK}
\begin{document}
\maketitle

\begin{abstract}
Classical strong external difference families (SEDFs) are much-studied combinatorial structures motivated by information security applications; it is conjectured that only one classical abelian SEDF exists with more than two sets.  Recently, non-disjoint SEDFs were introduced; it was shown that families of these exist with arbitrarily many sets.  We present constructions for both classical and non-disjoint SEDFs, which encompass all known non-cyclotomic examples for either type (plus many new examples) using a sequence-based framework.  Moreover, we introduce a range of new external difference structures (allowing set-sizes to vary, and sets to be replaced by multisets) in both the classical and non-disjoint case, and show how these may be applied to various communications applications. 
\end{abstract}

\begin{section}{Introduction}
\let\thefootnote\relax
Strong external difference families (SEDFs) are combinatorial structures consisting of disjoint families of sets in a group $G$ satisfying a strong uniformity condition on their external differences. 
These were introduced by Paterson and Stinson in \cite{PatSti}, motivated by an information security application to AMD codes and secret-sharing schemes.  In \cite{PatSti}, various related combinatorial objects were introduced, including generalised strong external difference families (GSEDFs) in which the set-sizes are allowed to vary.  These combinatorial objects have been well-studied, with SEDFs gaining particular notoriety since only one example with more than two sets is currently known (comprising $11$ sets in a group of size $243$), despite extensive investigation \cite{BaoJiWeiZha,HucJefNep, JedLi, PatSti, WenYanFuFen}.   SEDFs appear to be very restricted in other respects too, with no infinite families known for any fixed $\lambda$ greater than $1$, and very few constructions known in general.  Numerous non-existence results have been proved for them \cite{JedLi, LeuLiPra, LeuPra, MarSti}, and it is conjectured that no further abelian examples with more than 2 sets exist \cite{LeuLiPra}.  We shall refer to these structures (whose sets are disjoint) as the ``classical" versions.

Recently in \cite{HucNg}, non-disjoint SEDFs were introduced; these are defined analogously to classical SEDFs except that $0$ is treated as a difference just like any other group element, and hence the sets are non-disjoint.  Unlike the situation for classical SEDFs, it is shown in \cite{HucNg} that constructions exist with any number of sets, and that 2-set constructions exist with any value of $\lambda$.  

All non-disjoint SEDF constructions so far, satisfy the property that the external differences between any pair of sets satisfy the uniformity condition; we call these non-disjoint PSEDFs.  (All known classical SEDFs have an analogous structure, with the exception of the single example with $11$ sets).  Non-disjoint PSEDFs arise naturally in communications applications because in $\mathbb{Z}_v$ they correspond precisely to optical orthogonal codes (OOCs) in which the lowest-possible cross-correlation value $\lambda_c$ is obtained \cite{HucNg}.  Pairwise cross-correlation is also of interest in other communications settings, e.g. pairwise shift invariant protocol sequences \cite{ZhaShuWon}.

In this paper, we further investigate non-disjoint SEDFs and classical SEDFs, and provide a unified sequence-based framework for both.  We present a construction which encompasses all known non-disjoint SEDFs as special cases (and provides many new ones), and another construction which encompasses all known classical SEDF results not obtainable from cyclotomy.  The latter is useful for strong circular AMD codes and strong circular EDFs (SCEDFs), recently introduced in \cite{VeiSti} to construct non-malleable threshold schemes.  While the connection between sequences and difference structures is known, its constructive use in describing and building external difference families is novel.

We also extend the definition of non-disjoint PSEDFs and non-disjoint SEDFs in two directions.  Firstly, we allow distinct set-sizes, to obtain versions which we call ``generalised".  This is a precise analogue of the process by which a classical SEDF is extended to a classical GSEDF in \cite{PatSti}.  Again, in the non-disjoint setting we can obtain a wider range of parameters than is currently known for classical GSEDFs \cite{LuNiuCao, WenYanFuFen}.  Secondly, we allow the component sets to be replaced by multisets.  This may be viewed as an external analogue of the ``strong difference families" introduced in \cite{Bur}; the multiset versions also provide a useful technical tool in establishing non-multiset results.  

We demonstrate ways in which new non-disjoint SEDFs and GPSEDFs can be made from old, obtaining the first examples in non-cyclic abelian groups.  We also show that non-disjoint SEDFs and GSEDFs exist in non-abelian groups; there is currently just one known infinite non-abelian family of classical SEDFs (\cite{HucJefNep}). 

While non-disjoint PSEDFs in $\mathbb{Z}_v$ yield constant-weight OOCs, non-disjoint GPSEDFs yield variable-weight OOCs (VW-OOCs), first introduced in \cite{Yan}.  Our constructions (and adaptations of these) yield infinite families of variable-weight OOCs for a wide range of parameters; a lack of systematic constructions for VW-OOCs is noted in \cite{ChuYan}. Non-disjoint GPSEDFs in $\mathbb{Z}_v$ also provide examples of pairwise shift-invariant protocol sequences \cite{ZhaShuWon}.  

For the aid of the reader throughout this paper, we end the introductory section with Figure \ref{fig:SEDFdiagram}, a diagram indicating how the structures in this paper are related to each other.  If structure $A$ is higher in the diagram than structure $B$ and there is a line between them, then any example of structure $A$ is an example of structure $B$.  
\vspace{2mm}

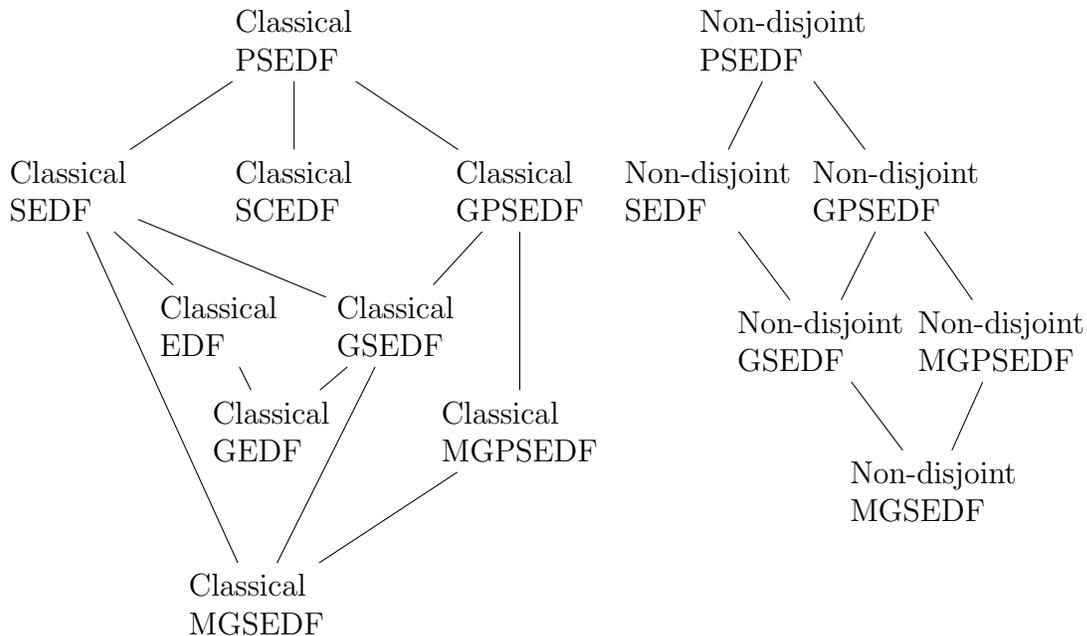
\begin{figure}\label{fig:SEDFdiagram}
\caption{Diagrams of the classical and non-disjoint external difference structures}
\begin{tikzpicture}
  \node (max) at (-0.5,4)  [align=left]{Classical \\ PSEDF};
  \node (a) at (-3.5,2) [align=left]{Classical \\ SEDF};
  \node (b) at (-0.5,2) [align=left]{Classical \\ SCEDF};
  \node (c) at (2.5,2) [align=left]{Classical \\ GPSEDF};
  %\node (d) at (-4,0)  [align=left]{Classical \\ MSEDF};
  \node (e) at (0.85,0.2) [align=left]{Classical \\ GSEDF};
  \node (f) at (2.5,-1.2)  [align=left]{Classical \\ MGPSEDF};
 \node (g) at (-1.5,0.2) [align=left]{Classical \\ EDF};
 \node (h) at (-0.8,-1.2) [align=left]{Classical \\ GEDF};
  \node (min) at (-1,-3.5)  [align=left]{Classical \\ MGSEDF};
  \draw (min) -- (a) -- (max) -- (b) 
  (min) -- (f) -- (c) -- (max)
  (h) -- (g) -- (a)
  (min) -- (e)
 (e) -- (h);
\draw[preaction={draw=white, -,line width=6pt}] (a) -- (e) -- (c);
 \node (max1) at (6,4)  [align=left]{Non-disjoint \\ PSEDF};
  \node (i) at (5,2) [align=left]{Non-disjoint \\ SEDF};
  \node (j) at (7.5,2) [align=left]{Non-disjoint \\GPSEDF};
  \node (k) at (6.5,0) [align=left]{Non-disjoint\\GSEDF};
  \node (l) at (8.9,0)  [align=left]{Non-disjoint \\ MGPSEDF};
  \node (min1) at (8,-2)  [align=left]{Non-disjoint \\ MGSEDF};
  \draw (min1) -- (k) -- (i) -- (max1) 
 (min1) -- (l) -- (j) -- (max1)
 (k) -- (j);
\end{tikzpicture}
\end{figure}

\end{section}

\begin{section}{Background}
\subsection{Definitions and context}
For subsets $A, B$ of a group $G$ (written additively), we define the multiset
 $\Delta(A,B)=\{a-b:a\in A, b \in B\}$.  In multiplicative notation we can write this as $\Delta(A,B) = \{ab^{-1}: a \in A, b \in B\}$. (In general we will use additive notation except in situations where multiplicative notation is more natural.) For $\lambda \in \mathbb{N}$ and $S \subseteq G$, we will write $\lambda S$ to denote the multiset comprising $\lambda$ copies of $S$ (except where this might cause confusion, e.g. in Section \ref{section:groups}, where different phrasing is used).

The concept of classical external difference family was first defined in \cite{Oga}, motivated by an application to AMD codes:

\begin{definition}\label{def:EDF}
Let $G$ be a group of order $v$ and let $m>1$.  A family of disjoint $k$-sets $\{A_1, \ldots, A_m\}$ in $G$ is a classical $(v,m,k,\lambda)$-EDF if the multiset equation
$ \bigcup_{i \neq j} \Delta(A_i,A_j)= \lambda(G \setminus \{0\})$ holds.
\end{definition}

The strong version was first defined in \cite{PatSti}, corresponding to a stronger security model. 
\begin{definition}\label{def:SEDF}
Let $G$ be a group of order $v$ and let $m>1$.  A family of disjoint $k$-sets $\{A_1, \ldots, A_m\}$ in $G$ is a classical $(v,m,k,\lambda)$-SEDF if, for each $1 \leq i \leq m$, the multiset equation
$ \bigcup_{j \neq i} \Delta(A_i,A_j)= \lambda(G \setminus \{0\})$ holds.
\end{definition} 

\begin{example}
The family of sets $\{\{0,1,2\},\{3,6,9\}\}$ form a classical $(10,2,3,1)$-SEDF in $\mathbb{Z}_{10}$.
\end{example}

Classical SEDFs have been much studied, with various constructions and non-existence results established.  Currently only one parameter set with $m>2$ is known, namely $(243,11,22,20)$; this was established independently in \cite{WenYanFuFen} and \cite{JedLi}. The only value of $\lambda$ for which infinite families of constructions are known is $\lambda=1$ (see \cite{PatSti} and \cite{HucJefNep}). Cyclotomic constructions are also known (see \cite{WenYanFuFen}).  Numerous non-existence results exist (see \cite{LeuLiPra} and \cite{HucNg} for recent summaries); it is conjectured that no further classical SEDFs with $m>2$ exist in abelian groups, and this is supported by computational evidence.

In \cite{PatSti}, the generalised versions were introduced:
\begin{definition}
Let $G$ be a group of order $v$ and $m \geq 1$. 
\begin{itemize}
\item A classical $(v,m; k_1,\ldots,k_m; \lambda)$-GEDF is a family of disjoint sets $A=\{A_1,\dots,A_m\}$ in $G$ such that
such that $|A_i| = k_i$ for $1 \leq i \leq m$ and the multiset equation $\cup_{i \neq j} \Delta(A_i, A_j) = \lambda (G \setminus \{0\})$ holds.
\item  A classical $(v,m;k_1,\dots,k_m;\lambda_1,\dots,\lambda_m$)-GSEDF is a family of disjoint sets $A=\{A_1,\dots,A_m\}$ in $G$ such that for each $1 \leq i \leq m,|A_i|=k_i$ and the multiset equation $\cup_{j \neq i} \Delta(A_i,A_j)=\lambda_i (G \setminus \{0\})$ holds.
\end{itemize}
\end{definition}
Classical GSEDFs have been studied in \cite{WenYanFuFen} and \cite{LuNiuCao}.  The structure of those partitioning $G$ or $G \setminus \{0\}$ is known, and there are some cyclotomic constructions with $m=2$ or by taking the collections of cyclotomic classes together with the zero singleton-set. It is known that for $m=3$ a classical $(v,m;k_1, k_2, k_2; \lambda_1,\lambda_2,\lambda_3)$-GSEDF does not exist when $k_1+k_2+k_3<v$, and for $k_1+k_2+k_3=v$ the parameters are very restricted, with either $\{k_1,k_2,k_3\}=\{1, \frac{v-1}{2}, \frac{v-1}{2}\}$ (examples known) or $\sqrt{v}<k_1<k_2<k_3$ (no examples known).

The following non-disjoint analogues of SEDFs were introduced in \cite{HucNg}, motivated by an application to optical orthogonal codes. While still satisfying extremely restrictive conditions, there is a much richer landscape of constructions for the non-disjoint versions.

\begin{definition}\label{def:nondisjointSEDF}
Let $G$ be a group of order $v$ and let $m > 1$. We say that a family of $k$-sets
$\{A_1, \dots, A_m\}$ in $G$ is a non-disjoint $(v, m, k, \lambda)$-SEDF (Strong External Difference Family) if, for each $1 \leq i \leq m$, the multiset equation $\cup_{j \neq i} \Delta(A_i,A_j)=\lambda G$ holds. 
\end{definition}

\begin{definition}
Let $G$ be a group of order $v$ and let $m > 1$. We say that a family of $k$-sets
$\{A_1, \dots, A_m\}$ in $G$ is a non-disjoint $(v, m, k, \lambda)$-PSEDF (Pairwise Strong External Difference Family) if, for each $1 \leq i \neq j \leq m$, the multiset equation $\Delta(A_i,A_j)=\lambda G$ holds. 
\end{definition}

The following basic results hold ((i) and (ii) are noted in \cite{HucNg}).
\begin{proposition}
\begin{itemize}
\item[(i)] A non-disjoint $(v,m,k,\lambda)$-PSEDF is a non-disjoint $(v,m,k,(m-1)\lambda)$-SEDF.
\item[(ii)] A non-disjoint $(v,2,k,\lambda)$-SEDF is a non-disjoint $(v,2,k,\lambda)$-PSEDF.
\item[(iii)] A collection of sets $\{A_1,\ldots,A_m\}$ in a group $G$ is a  non-disjoint $(v,m,k,\lambda)$-PSEDF if and only if, for each  $1 \leq i \neq j \leq m$, $\{A_i,A_j\}$ is a non-disjoint $(v,2,k,\lambda)$-PSEDF.
\end{itemize}
\end{proposition}

\begin{example}
The family of sets $\{ \{0,1,2\}\}, \{0,3,6\} \}$ form a non-disjoint $(9,2,3,1)$-PSEDF (and so a non-disjoint $(9,2,3,1)$-SEDF) in $\mathbb{Z}_9$. 
\end{example}

In this paper, we introduce the non-disjoint analogue of the classical GSEDF:

\begin{definition}
Let $G$ be a group of order $v$ and let $m \geq 1$. A non-disjoint $(v,m,k_1,\dots,k_m;\lambda_1,\dots,\lambda_m)$-GSEDF is a family of sets, $A=\{A_1,\dots,A_m\}$ in $G$ such that for each $1 \leq i \leq m,|A_i|=k_i$ and the following equation holds:
\[\cup_{j \neq i} \Delta(A_i,A_j)=\lambda_iG\]
\end{definition}

We present the pairwise version of this structure:
\begin{definition}
Let $G$ be a group of order $v$. We say that a family of sets $A = \{A_1,\dots,A_m\}$ in $G$ is a non-disjoint $(v,m,k_1,\dots,k_m)$-GPSEDF if $|A_i|=k_i$ ($1 \leq i \leq m$) and the following equation holds for all $1 \leq i \neq j \leq m$:
\[\Delta(A_i,A_j)=\lambda_{i,j}G \text{ where } \lambda_{i,j}=\frac{k_ik_j}{v}\]
Note that $\lambda_{j,i}=\lambda_{i,j}$.
\end{definition}

\begin{definition}
We define the $\lambda$-matrix of a non-disjoint ($v,m,k_1,\dots,k_m$)-GPSEDF to be the $m \times m$ matrix $[a_{ij}]$, with $a_{ii}=0$ and $a_{ij}=\lambda_{i,j}=\frac{k_ik_j}{v}$ ($1 \leq i,j \leq m$).
\end{definition}

The following result presents a known classical GSEDF construction, and an analogous new non-disjoint GSEDF construction.
\begin{proposition}\label{GSEDF:basic}
\begin{itemize}
\item[(i)] In $\mathbb{Z}_{k_1 k_2+1}$, the sets $A_1=\{0,1,2 \ldots, k_1-1\}$ and $A_2=\{k_1,2k_1,\ldots, k_1 k_2\}$ form a classical  $(k_1 k_2+1, 2; k_1, k_2; 1,1)$-GSEDF.
\item[(ii)] In $\mathbb{Z}_{k_1 k_2}$, the sets $B_1=\{0,1,2 \ldots, k_1-1\}$ and $B_2=\{0, k_1,2k_1,\ldots, k_1 (k_2-1)\}$ form a non-disjoint $(k_1 k_2, 2; k_1, k_2; 1,1)$-GSEDF which is a non-disjoint $(k_1 k_2,2,k_1,k_2)$-GPSEDF.
\end{itemize}
\end{proposition}
\begin{proof}
Part (i) is Theorem 4.4 of \cite{LuNiuCao}.  For (ii), direct checking shows that $\Delta(B_1,B_2)=\Delta(B_2,B_1)=\mathbb{Z}_{k_1 k_2}$.
\end{proof}

\begin{proposition}
\begin{itemize}
\item[(i)] A non-disjoint $(v,m,k,\lambda)$-SEDF is a non-disjoint\\ $(v,m; k,\dots,k;\lambda,\dots,\lambda)$-GSEDF (with $m$ $\lambda$s and $m$ $k$s).
\item[(ii)] A non-disjoint $(v,m,k,\lambda)$-PSEDF is a non-disjoint $(v,m,k,\dots,k)$-GPSEDF with $m \times m$ $\lambda$-matrix $M=[a_{ij}]$, where
$a_{ii}=0$ and $a_{ij}=\frac{k^2}{v}$ for $i \neq j$.
\item[(iii)] A non-disjoint $(v,2;k_1,k_2;\lambda_1,\lambda_2)$-GSEDF is a non-disjoint $(v,2,k_1,k_2)$-GPSEDF. (Here $\lambda_1=\lambda_2$).
\item[(iv)] A non-disjoint $(v,m,k_1,\dots,k_m)$-GPSEDF is a non-disjoint\\ $(v,m,k_1,\dots,k_m;\sum_{j \neq 1} \lambda_{1,j},\dots,\sum_{j \neq m} \lambda_{m,j})$-GSEDF.
\item[(v)] A collection of sets $\{A_1,\ldots,A_m\}$ in a group $G$ is a non-disjoint $(v,m,k_1,\ldots,k_m)$-GPSEDF if and only if $\{A_i,A_j\}$ is a non-disjoint $(v,2,k_i,k_j)$-GPSEDF for each $i \neq j$.
\end{itemize}
\end{proposition}

\begin{example}
In $\mathbb{Z}_{12}$, the sets $\{0,1,2,3,4,5\}$, $\{0,1,2,6,7,8\}$ and $\{0,1,3,4,6,7,9,10\}$ form a non-disjoint $(12,3,6,6,8)$-GPSEDF with $\lambda_{1,2}=3=\lambda_{2,1}, \lambda_{2,3}=4=\lambda_{3,2}, \lambda_{1,3}=4=\lambda_{3,1}$.  Hence these sets also form a non-disjoint $(12,3; 6,6,8; 7,7,8)$-GSEDF.  Note this is in contrast to the situation with classical GSEDFs where $|\{k_1,k_2,k_3\}|=2$ implies $\{k_1,k_2,k_3\}=\{1, \frac{v-1}{2}, \frac{v-1}{2}$.)
\end{example}

\begin{example}
In $\mathbb{Z}_{30}$, the sets $\{0,1,2,3,4,5,6,7,8,9,10,11,12,13,14,15,16,17\}$, \\ $\{0,1,2,6,7,8,12,13,14,18,19,20,24,25,26\}$ and \\ $\{0,1,3,4,6,7,9,10,12,13,15,16,18,19,21,22,24,25,27,28\}$ form a non-disjoint $(30,3,18,15,20)$-GPSEDF with $\lambda_{1,2}=9=\lambda_{2,1}, \lambda_{2,3}=10=\lambda_{3,2}, \lambda_{1,3}=12=\lambda_{3,1}$.  Hence these sets also form a non-disjoint $(30,3; 18,15,20; 21,19,22)$-GSEDF.  Note that no examples of classical GSEDFs are known with all set-sizes different.
\end{example}

We introduce a multiset version of these structures:

\begin{definition}
Let $G$ be a group of order $v$. Let $m>1$. We say a family of multisets $A=\{A_1,\dots,A_m\}$ in $G$ is a:
\begin{itemize}
\item[(i)] non-disjoint ($v,m,k_1,\dots,k_m,\lambda_1,\dots,\lambda_m$)-MGSEDF (multiset generalised strong external difference family) if for each $1 \leq i \leq m$, $|A_i|=k_i$, and the following equation holds:
\[\cup_{j\neq i} \Delta(A_i,A_j) = \lambda_i G\]
\item[(ii)] non-disjoint ($v,m,k_1,\dots,k_m$)-MGPSEDF (multiset generalised pairwise strong external difference family) if $|A_i|=k_i$ ($1 \leq i \leq m$) and the following equation holds for all $1 \leq i \neq j \leq m$:
\[\Delta(A_i,A_j)= \lambda_{i,j}G\text{ where } \lambda_{i,j}=\frac{k_ik_j}{v}\]
\end{itemize}
\end{definition}

Since any set without repeated elements may be viewed as a multiset, this definition includes the original versions (eg a non-disjoint GSEDF is an example of a non-disjoint MGSEDF).

In this paper, we will write a multiset $M$ as $\{a_1,b_1;a_2,b_2;\dots;a_n,b_n\}$, where this indicates that $M$ comprises  $b_1$ occurrences of $a_1$, $b_2$ occurrences of $a_2$ and so on.

\begin{example}
The multisets 
\begin{itemize}
    \item $\{0,2;1,2;2,2;3,2;4,2;5,2\}$, 
    \item $\{0,3;1,3;2,3;3,2;4,2;5,2;6,3;7,3;8,3;9,2;10,2;11,2\}$
    \item $\{0,4;1,5;3,4;4,5;6,4;7,5;9,4;10,5\}$ 
    \end{itemize}
    form a non-disjoint ($12,3,12,30,36$)-MGPSEDF in $\mathbb{Z}_{12}$. Here $\lambda_{1,2}=30=\lambda_{2,1}$, $\lambda_{1,3}=36=\lambda_{3,1}$ and $\lambda_{2,3}=90=\lambda_{3,2}$
\end{example}

Finally, we define classical versions of these new concepts:
\begin{definition}
\begin{itemize}
\item[(i)] A \emph{classical} $(v,m,k,\lambda)$-PSEDF is a family of disjoint $k$-sets $\{A_1, \dots, A_m\}$ in $G$ such that, for each $1 \leq i \neq j \leq m$, the multiset equation $\Delta(A_i,A_j)=\lambda (G \setminus \{0\})$ holds. 
\item[(ii)]  A \emph{classical} $(v,m,k_1,\dots,k_m)$-GPSEDF is a family of disjoint $k$-sets $A = \{A_1,\dots,A_m\}$ in $G$ such that $|A_i|=k_i$ ($1 \leq i \leq m)$ and the following equation holds for all $1 \leq i \neq j \leq m$:
\[\Delta(A_i,A_j)=\lambda_{i,j}(G \setminus \{0\}) \text{ where } \lambda_{i,j}=\frac{k_ik_j}{v-1}\]
\item[(iii)] The \emph{classical} MGSEDF and MGPSEDF are defined in the same way as the classical GSEDF and GPSEDF respectively, except that $A=\{A_1,\ldots, A_m\}$ is now a family of \emph{multisets} rather than a family of sets in each case.
\end{itemize}
\end{definition}

\begin{example}
In $\mathbb{Z}_{13}$, the multisets $$\{1,2;3,2;4,2;9,2;10,2;12,2\}, \{2,3;5,3;6,3;7,3;8,3;11,3\}$$ 
form a classical ($13,2,12,18$)-MGPSEDF and hence also a classical ($13,2,12,18$)-MGSEDF with $\lambda_{1,2}=18=\lambda_{2,1}$.
\end{example}

In \cite{VeiSti}, the following combinatorial object is introduced.

\begin{definition}
Let $G$ be an abelian group of order $v$. A (classical) $(v,m,k;\lambda)$-strong circular external difference family (or (classical)) $(v, m, k; \lambda)$-SCEDF) is a set of $m$ disjoint $k$-subsets, $A = \{A_0, \ldots, A_{m-1}\}$ of $G$, such that the following multiset equation holds for every $j$ $(0 \leq j \leq m-1)$:
$$ \Delta(A_{j+1 \mod m}, A_j ) = \lambda(G \setminus \{0\})$$
\end{definition}

Clearly any classical $(v,2,k,\lambda)$-SEDF is an example of a classical $(v,2,k;\lambda)$-SCEDF.  The following structural theorem is presented in \cite{PatSti2}.
\begin{theorem}
Let $G$ be an abelian group of order $v$. A set of $m$ disjoint $k$-subsets of $G$, say $A=\{A_0, \ldots,A_{m-1} \}$ is a classical $(v,m,k;\lambda)$-SCEDF if and only if
$\{A_i, A_{i+1 \mod m}\}$ is a classical $(v, 2, k, \lambda)$-SEDF for all $0 \leq i \leq m-1$.
\end{theorem}

We observe that the existence of a classical $(v,m,k,\lambda)$-PSEDF implies the existence of a classical $(v,m,k;\lambda)$-SCEDF.  However, it is proved in Theorem 6 of \cite{WuYanFen} that no classical abelian $(v,m,k;\lambda)$-SCEDF can exist for $m \geq 3$.  This leads to the following non-existence result for classical PSEDFs:
\begin{theorem}
In an abelian group of order $v$, there does not exist a classical $(v,m,l,\lambda)$-PSEDF with $m \geq 3$.
\end{theorem}

The next result relates the multiset/non-multiset structures.
\begin{theorem}
Let $A = \{A_1,\dots,A_m\}$ be a family of sets in a group $G$.
\begin{itemize}
\item[(i)] $A$ is a non-disjoint GSEDF $\Leftrightarrow \{A_i, \cup_{j \neq i} A_j\}$ is a non-disjoint MGPSEDF for all $1 \leq i \leq m$.
\item[(ii)] $A$ is a classical GSEDF $\Leftrightarrow \{A_i, \cup_{j \neq i} A_j\}$ is a classical GSEDF for all $1 \leq i \leq m$
\end{itemize}
\end{theorem}
\begin{proof}
For any $1 \leq i \leq m$, observe that the multiset union $\uplus_{j \neq i} \Delta(A_i, A_j)=\Delta(A_i, \uplus_{j \neq i} A_j)$. For (ii), note that $\cup_{j \neq i} A_j$ is a set rather than a multiset by disjointness.
\end{proof}

\subsection{Basic structural results}
In this section, we present some useful structural results which will be used in the paper. 

We will use the following notation: for a set $S$, $\overline{S}=G \setminus S$ will denote the complement of $S$ in $G$.  For $g \in G$ and $S \subseteq G$, we denote $g+S=\{g+s: s \in S\}$.  For $A,B \subseteq G$, $A \uplus B$ denotes the multiset union of $A$ and $B$.

The next theorem allows us to obtain new examples from old in various natural ways.

\begin{theorem} \label{thm:combinations}
Let $A=\{A_1,\ldots,A_m\}$ be a non-disjoint $(v,m,k_1,k_2,\ldots,k_m)$-GPSEDF in a group $G$ of order $v$, written additively (we do not assume $G$ is abelian).  Then:
\begin{itemize}
\item[(i)] For $1 \leq j \leq m$, $\{A_1,A_2,\ldots, \overline{A_j}, \ldots, A_m\}$ is a non-disjoint $(v,m,k_1,k_2,\ldots, v-k_j, \ldots, k_m)$-GPSEDF.
\item[(ii)] $\{\overline{A_1}, \overline{A_2},\dots, \overline{A_m} \}$ is a non-disjoint $(v,m,v-k_1,v-k_2,\dots,v-k_m)$-GPSEDF in $G$.
\item[(iii)] For $g \in G$, $\{ A_1,\dots, A_{m-1}, g + A_m \}$ is a non-disjoint $(v,m,k_1,k_2,\ldots,k_m)$-GPSEDF.
\item[(iv)] Let $B=\{A_1,\dots,A_{m-1},B_m\}$ be a non-disjoint $(v,m,k_1,\dots, k_{m-1}, l_m)$-MGPSEDF in $G$ (i.e. $B$ has the same sets as $A$ except for one set). Then 
$$C=\{A_1,A_2,\dots,A_{m-1},A_m \uplus B_m\}$$ is a  non-disjoint $(v,m,k_1,\dots, k_{m-1},k_m+ l_m)$-MGPSEDF.   If $A$ and $B$ are non-disjoint GPSEDFs and $A_m$ and $B_m$ are disjoint sets, then $C$ is a non-disjoint GPSEDF.
\item[(v)]  Let $g_1,\dots,g_n \in G$ for some $n \in \mathbb{N}$. Then $\{A_1,\dots,A_{m-1},\uplus_{i=1}^n (g_i+A_m))$ is a non-disjoint $(v,m,k_1,\dots,k_{m-1},nk_m)$-MGPSEDF. If all the sets $\{g_i+A_m\}$ for $1 \leq i \leq n$ are pairwise disjoint, then this construction is in fact a non-disjoint GPSEDF. 
\end{itemize}
\end{theorem}
\begin{proof}
(i) Let $i \neq j$, $1 \leq i \leq m$.  Then
\begin{align*}
\Delta(\overline{A_j},A_i) &= \Delta(G \setminus A_j, A_i)\\
&=\{g-b: g \in G, g \notin A_j, b \in A_i\}\\
&=\{g-b: g \in G, b \in A_i\}\setminus\{g-b: g \in A_j, b \in A_i\}\\
&=|A_i|G\setminus \lambda_{j,i} G\\
&= (k_i-\lambda_{j,i})G
\end{align*}
By a similar argmuent, $\Delta(A_i,\overline{A_j})=(k_i-\lambda_{i,j})G$.  All other external differences are the same as in $A$. The result follows.\\
(ii) This follows by repeating the method from (i) $m$ times for each set in $A$.\\
(iii) Let $g \in G$ and let $1 \leq j \leq m-1$.  Consider
\begin{align*}
    \Delta(g+A_m,A_j)&=\{g+a-c:a \in A_m, c \in A_j\}\\
    &=g+\{a-c:a \in A_m, c \in A_j\}\\
&=g+ \Delta(A_m,A_j)\\
    &= \lambda_{m,j} G
\end{align*}
By a similar argument, $\Delta(A_j,g+A_m)= \{c+(-a-g): c \in A_j,a \in A_m\}= \Delta(A_j,A_m)+(-g)=\lambda_{j,m}G$ and the result follows. \\
(iv)  Let $1 \leq i \leq m-1$.  We show that $\Delta(A_i,A_m \uplus B_m)=\lambda_{i,m}$  for some $\lambda_{i,m} \in \mathbb{N}$.
\begin{align*}
\Delta(A_i,A_m \uplus B_m)&=\{a_i-a_m:a_i \in A_i, a_m \in A_m \uplus B_m\}\\
&=\{a_i-a_m:a_i \in A_i, a_m \in A_m\} \uplus \{a_i-b_m: a_i \in A_i, a_m \in B_m\}\\
&=\Delta(A_i,A_m) \uplus \Delta(A_i,B_m)\\
&=\lambda_{i,m} G +\mu_{i,m} G \quad \text{for some } \lambda_{i,m}, \mu_{i,m} \in \mathbb{N}\\
&=(\lambda_{i,m} + \mu_{i,m}) G
\end{align*}
By a similar argument, $\Delta(A_m \uplus B_m,A_i) = (\lambda_{m,i} + \mu_{m,i}) G$. Since  $|A_m \uplus B_m| = |A_m| + |B_m|$, the result follows.  If $A$ and $B$ are non-disjoint GPSEDFs and $A_m$ and $B_m$ are disjoint then $A_m \uplus B_m=A_m \cup B_m$ and so $C$ is also a non-disjoint GPSEDF.\\
(v) Applying parts (iii) and (iv), $\{A_1,\dots,A_{m-1},\uplus_{i=1}^n (g_i + A_m)\}$ is a non-disjoint MGSEDF, which is a non-disjoint GPSEDF if $A_m$ and $B_m$ are disjoint.
\end{proof}

\begin{theorem}\label{thm:cosets}
Let $G$ be a finite abelian group and let $H$ be a subgroup of $G$.\\
If $A=\{a_1,\ldots,a_r\}$ and $B=\{b_1,\ldots,b_s\}$ form a non-disjoint $(|G/H|,2,r,s)$-GPSEDF in $G/H$, then $A'=\cup_{i=1}^r (a_i+H)$ and $B'=\cup_{j=1}^s (b_j+H)$ form a non-disjoint $(|G|, 2, r|H|,s|H|)$-GPSEDF in $G$.  
\end{theorem}
\begin{proof}
 Let $1 \leq i \leq r, 1 \leq j \leq s$. 
As a multiset, $\Delta(a_i+H, b_j+H)=|H|( (a_i-b_j)+H)$. So $\Delta(A',B')=\cup_{i=1}^r \cup_{j=1}^s \Delta(a_i+H, b_j+H)=|H|\cup_{i=1}^r \cup_{j=1}^s ((a_i-b_j)+H)$.  
 
If $\{A,B\}$ is a non-disjoint GPSEDF in $G/H$ then the multiset $\cup_{i=1}^r \cup_{j=1}^s (a_i-b_j)=\lambda (G/H)$ where $\lambda=\frac{rs}{|G/H|}$ and we have that $\Delta(A',B')$ comprises $\lambda|H|$ copies of each $x+H$ ($x \in G/H$), i.e. $\lambda|H|$ copies of $G$. 
\end{proof}

A key technique in this paper is the use of binary sequences.  We briefly outline the relationship between sets and sequences.  The binary sequence case is well-known and summarized in our context in \cite{HucNg}; we give it here for completeness, and also extend this to multisets.

\begin{definition}\label{sequenceset}
\begin{itemize}
\item[(i)] Let $A$ be a $k$-subset of $\mathbb{Z}_v$ ($v \in \mathbb{N}$).  We associate to $A$ a binary sequence $X=(x_t)_{t=0}^{v-1}$ of weight $k$, where $x_t=1$ if $t \in A$ and $x_t=0$ if $t \notin A$.
\item[(ii)] Let $X$ be a binary sequence $(x_t)_{t=0}^{v-1}$ of weight $k$.  We associate to $X$ a $k$-subset $A$ of $\mathbb{Z}_v$, where $t \in A \Leftrightarrow x_t = 1$.
\item[(iii)] Let $A$ be a multiset of elements of $\mathbb{Z}_v$ ($v \in \mathbb{N}$).  We associate to $A$ a sequence $X=(x_t)_{t=0}^{v-1}$ of non-negative integers, where $x_t$ equals the number of occurrences of $t$ in $A$.
\item[(iv)] Let $X=(x_t)_{t=0}^{v-1}$ be sequence of non-negative integers.  We associate to $X$ a multiset $A$ of elements of $\mathbb{Z}_v$, where $t$ occurs $i$ times in $A \Leftrightarrow x_t = i$.
\end{itemize}
\end{definition}

The following result was given in \cite{HucNg}.

\begin{theorem} \label{BinSeqSets}
Let $X_A=(x_t)_{t=0}^{v-1}$ and $X_B=(y_t)_{t=0}^{v-1}$ (with indices taken modulo $v$) be two binary sequences corresponding to subsets $A$ and $B$ in $\mathbb{Z}_v$. Then:
\begin{itemize}
\item[(i)] For $\delta \in \{0,\dots,v-1\}$, $\sum_{t=0}^{v-1} x_ty_{t+\delta}$ equals the number of occurrences of $\delta$ in $\Delta(B,A)$.
\item[(ii)] $\sum_{t=0}^{v-1} x_ty_{t+\delta} = \lambda$ for all $0 \leq \delta \leq v-1 \Leftrightarrow \Delta(B,A)=\lambda \mathbb{Z}_v$.
\end{itemize}
\end{theorem}

In fact we can generalize this to the multiset case:
\begin{theorem}\label{thm:multi_seq}
Let $X_A = (x_i)_{i=0}^{v-1}$, $X_B=(y_i)_{i=0}^{v-1}$, $(x_i,y_i  \in \mathbb{N} \cup \{0\})$, with indices taken modulo v, be the sequences corresponding to multisets $A$ and $B$ of elements of $\mathbb{Z}_v$. Then:
\begin{itemize}
\item[(i)] For $\delta \in \{0,\dots,v-1\}$, $\sum_{t=0}^{v-1} x_ty_{t+\delta}$ equals the number of occurrences of $\delta$ in $\Delta(B,A)$.
\item[(ii)] $\sum_{t=0}^{v-1} x_ty_{t + \delta} = \lambda$ for all $0 \leq \delta \leq v-1$ if and only if $\Delta(B,A)=\lambda \mathbb{Z}_n$.
\end{itemize}
\end{theorem}
\begin{proof}
We prove (i); (ii) is then immediate. For fixed $\delta, t \in \mathbb{Z}_v$, $x_ty_{t+\delta}$ is the product of the number of occurrences of $t$ in $A$ and the number of occurrences of $t+\delta$ in $B$. This is the number of occurrences of $\delta$ as a difference between all the $(t+\delta)$s in $B$ and $t$s in $A$. Hence for a given $\delta$, the number of occurrences of $\delta$ in $\Delta(B,A)$ is the sum of this quantity over all $t$ in the sequence, as required.
\end{proof}

\end{section}

\begin{section}{New results on non-disjoint SEDFs and GSEDFs using sequence methods}\label{section:seq}

In \cite{HucNg}, two main constructions were presented: one for a $2$-set non-disjoint PSEDF with any value of $\lambda$ and another for a non-disjoint PSEDF with any number of sets.  Both were obtained using sequence methods. We summarize the statements of these below.

\begin{theorem}\label{HucNgpaper}
\begin{itemize}
\item[(i)] For any $a,r \in \mathbb{N}$, there exists a non-disjoint $(r a^2,2,ra,r)$-PSEDF in $\mathbb{Z}_{ra^2}$ and hence a non-disjoint $(r a^2,2,ra,r)$-SEDF in $\mathbb{Z}_{ra^2}$.
\item[(ii)] Let $N>1$.  There exists a non-disjoint $(2^N, N, 2^{N-1}, 2^{N-2})$-PSEDF in $\mathbb{Z}_{2^N}$ and hence a non-disjoint $(2^N,N,2^{N-1},(N-1)2^{N-2})$-SEDF in $\mathbb{Z}_{2^N}$.
\end{itemize}
\end{theorem}

The following key theorem generalizes these constructions, and yields many new non-disjoint GPSEDFs. We will later discuss which non-disjoint PSEDFs are obtainable from the theorem.

We use the following notation from \cite{Handbook} for a binary sequence $X=(x_t)_{t=0}^{v-1}$: a contiguous subsequence of length $r$ in $X$ is called a \emph{substring} of length $r$; a substring consisting entirely of $0$'s (not preceded or succeeded by the same symbol) is a \emph{gap}, and an analogous substring of $1$'s is a \emph{run}.
We call $X+s=(x_{i+s})_{i=0}^{v-1}$ a
(cyclic) shift of $X$ by $s$ places (indices taken modulo $v$).

\begin{theorem}\label{GPSEDFblock}
Let $m \geq 2$. Let $(a_0, a_1,\dots,a_m)$ be a sequence of positive integers such that for all $0 \leq i \leq m-1$, $a_i = b_{i+1}a_{i+1}$ for some $b_{i+1} (\geq 2) \in \mathbb{Z}$. Then, for all possible $1 \leq \eta_i \leq b_i-1$ there exists a non-disjoint $$(a_0,m,\eta_1a_1,\eta_2a_2\frac{a_0}{a_1},\eta_3a_3\frac{a_0}{a_2},\dots,\eta_ma_m\frac{a_0}{a_{m-1}})-\mbox{GPSEDF in } \mathbb{Z}_{a_0}.$$ 
The $\lambda$-matrix of this non-disjoint GPSEDF is $[a_{ij}]_{m\times m}$, where $a_{ij}= 0$ if $i=j$ and $a_{ij}=\lambda_{i,j}=\frac{a_0\eta_ia_i\eta_ja_j}{a_{i-1}a_{j-1}}$ otherwise.
\end{theorem}
\begin{proof}
Let $(a_0, a_1,\dots,a_m)$ ($m \geq 2$) be a sequence of positive integers such that for all $0 \leq i \leq m-1$, $a_i = b_{i+1}a_{i+1}$ for some $b_{i+1} (\geq 2) \in \mathbb{Z}$. Then fix $1 \leq \eta_i \leq b_i-1$ for all $1 \leq i \leq m$.

Firstly, we note that $a_j|a_i$ for all $i \leq j$ since $a_j|a_{j-1}|a_{j-2}\dots|a_i$.

Define $X_i=(x_t)_{t=0}^{a_0-1}$ for all $i = 1,\dots,m$, where:
\[x_t =\begin{cases}
    1,& 0\leq t \leq \eta_ia_i-1\\
    0,& \eta_ia_i\leq t \leq a_ib_i-1=a_{i-1}-1\\
    x_{t-a_ib_i}, & a_ib_i \leq t \leq a_0 -1\quad (a_ib_i=a_{i-1})
    \end{cases}\]
So we obtain:
\[X_i = \overbrace{\overbrace{\underbrace{\underbrace{1\dots1}_{a_i}\underbrace{1\dots1}_{a_i}\dots\underbrace{1\dots1}_{a_i}}_{\eta_i \text{copies}}\underbrace{\underbrace{0\dots0}_{a_i}\underbrace{0\dots0}_{a_i}\dots\underbrace{0\dots0}_{a_i}}_{(b_i-\eta_i) \text{copies}}}^{\text{length }a_ib_i=a_{i-1}}\dots\overbrace{\underbrace{\underbrace{1\dots1}_{a_i}\underbrace{1\dots1}_{a_i}\dots\underbrace{1\dots1}_{a_i}}_{\eta_i \text{copies}}\underbrace{\underbrace{0\dots0}_{a_i}\underbrace{0\dots0}_{a_i}\dots\underbrace{0\dots0}_{a_i}}_{(b_i-\eta_i) \text{copies}}}^{\text{length }a_ib_i=a_{i-1}}}^{\frac{a_0}{a_{i-1}}\text{ repeats}}\]

For all $1 \leq i \leq m$, we let $A_i$ be the set in $\mathbb{Z}_{a_0}$ corresponding to $X_i$ and we let $A = \{A_1,A_2,\dots,A_m\}$.

Since $X_i$ repeats itself every $a_{i-1}$ entry and the first substring of $X_i$ with length $a_{i-1}$ has weight $\eta_ia_i$ so does every other length $a_{i-1}$ substring of $X_i$. So any substring of length $ca_{i-1}$, ($c \in \mathbb{N}$) has weight $c\eta_ia_i$. In particular, $X_i$ has weight $\eta_ia_i\frac{a_0}{a_{i-1}}$ and so $|A_i|=\eta_ia_i\frac{a_0}{a_{i-1}}$. 
%We also note that there are $m$ such $X_i$s we have set up and so $m$ sets in $A = \{A_1,\dots,A_m\}$ by definition.\\

Let $i,j \in \{1,\dots,m\}$, $i < j$ and $\delta \in \{0,\dots,a_0-1\}$. We let $X_i = (x_t)_{t=0}^{a_0-1}$ and $X_j = (y_t)_{t=0}^{a_0-1}$. We consider the number of occurrences of $\delta$ in $\Delta(A_i,A_j)$ ($=\Delta(A_j,A_i)$ since $\mathbb{Z}_{a_0}$ is abelian), which by Theorem \ref{BinSeqSets} is equal to:
\begin{align*}
\sum_{t=0}^{a_0-1} x_ty_{t+\delta} &= \frac{a_0}{a_{i-1}}\sum_{t=0}^{a_{i-1}-1} x_ty_{t+\delta} &\text{($X_i$ and $X_j$ repeat every $a_{i-1}$ entry)}\\
&=\frac{a_0}{a_{i-1}}\sum_{t=0}^{a_{i}\eta_i-1} y_{t+\delta} &\text{($x_t = 1$ for $0 \leq t \leq \eta_ia_i-1$, $x_t = 0$ for $\eta_ia_i \leq t \leq a_{i-1}$)}\\
&=\frac{a_0}{a_{i-1}}\frac{a_i\eta_i}{a_{j-1}}\sum_{t=0}^{a_{j-1}-1} y_{t+\delta} &\text{($X_j$ repeats every $a_{j-1}$ entry, which divides $a_i\eta_i$)}\\
&=\frac{a_0a_i\eta_ia_j\eta_j}{a_{i-1}a_{j-1}} &\text{(weight of $a_{j-1}$ substring of $X_j$)}
\end{align*}

This is true for all $\delta \in \{0,\dots,a_0-1\}$ so $\Delta(A_i,A_j) = \Delta(A_j,A_i) =\frac{a_0a_i\eta_ia_j\eta_j}{a_{i-1}a_{j-1}}\mathbb{Z}_{a_0}$.  Hence $\Delta(A_i,A_j) = \frac{a_0a_i\eta_ia_j\eta_j}{a_{i-1}a_{j-1}}\mathbb{Z}_{a_0}$ for all $1 \leq i \neq j \leq m$, and thus $A = \{A_1,\dots,A_m\}$ forms a non-disjoint ($a_0,m,\eta_1a_1,\eta_2a_2\frac{a_0}{a_1},\eta_3a_3\frac{a_0}{a_2},\dots,\eta_ma_m\frac{a_0}{a_{m-1}}$)-GPSEDF with $\lambda$-matrix $[a_{ij}]_{m \times m}$, where $a_{ii} = 0$ if $i=j$ and $a_{ij}=\lambda_{i,j}=\frac{a_0\eta_ia_i\eta_ja_j}{a_{i-1}a_{j-1}}$ otherwise.
\end{proof}

We now present an equivalent form of this construction, which better shows the achievable parameters. 

\begin{corollary} \label{GPSEDFblockparameterfriendly}
Let $m \geq 2$.
Given any $(c_0,c_1,\ldots,c_m)$ with $c_0\ldots,c_{m-1} (\geq 2) \in \mathbb{Z}$ and $c_m \in \mathbb{N}$, and any  $1 \leq d_i \leq c_{i-1}-1$ ($1 \leq i \leq m$), there exists a non-disjoint
$$(v,m,d_1\frac{v}{c_0},d_2\frac{v}{c_1},\dots,d_m\frac{v}{c_{m-1}})-\mbox{GPSEDF}$$ in $\mathbb{Z}_v$, where $v=c_0 c_1\dots c_m$.
\end{corollary}
\begin{proof}
In Theorem \ref{GPSEDFblock}, we have $a_i=c_ic_{i+1} \dots c_m$ for all $0 \leq i \leq m$, where $b_{t+1}=c_t$ for $0 \leq t \leq m-1$ and $a_m=c_m$.  Note that $c_0 \ldots c_{i-1}=\frac{a_0}{a_i}$ and $c_{i+1} \ldots c_m= a_{i+1}$.
%$c_{i-1}=b_i$ and $c_m=a_m$.
\end{proof}

We now consider the circumstances under which Theorem \ref{GPSEDFblock} yields a non-disjoint PSEDF, using Corollary \ref{GPSEDFblockparameterfriendly}. This occurs when all the set sizes are the same, i.e. when $d_i \frac{v}{c_{i-1}}=d_j \frac{v}{c_{j-1}}$ for all $1 \leq i,j \leq m$.

\begin{theorem}\label{generalPSEDF}
Let $m \geq 2$ and let $(d_1,\dots,d_m),(c_0,\dots,c_m)$ be sequences of not necessarily distinct natural numbers such that $\frac{d_i}{c_{i-1}} = z$ for all $1 \leq i \leq m$, for some $z \in \mathbb{Q}$ with $0 < z < 1$.  (In other words, take $m$ expressions for $z$ as a fraction). Then there exists a non-disjoint ($c_0\dots c_m, m,c_0\dots c_mz,c_0\dots c_mz^2$)-PSEDF in $\mathbb{Z}_{c_0 \dots c_m}$. 
\end{theorem}

Note the condition $0 < z < 1$ is necessary since Corollary \ref{GPSEDFblockparameterfriendly} requires that $d_i < c_{i-1}$ for all $1 \leq i \leq m$.

Observe that Theorem \ref{HucNgpaper}(i) is the special case of Theorem \ref{generalPSEDF} with $c_0=c_1=a$, $c_2=r$ and $d_1=d_2=1$; while Theorem \ref{HucNgpaper}(ii) is the special case of Theorem \ref{generalPSEDF} with $c_0=c_1= \cdots=c_{N-1}=2$, $c_N=1$ and $d_1=d_2= \cdots d_N =1$.

We next present particular infinite families of non-disjoint PSEDFs arising from this construction.

\begin{corollary} \label{generalPSEDFsparameterfriendly}
    \begin{itemize}
        \item[(i)] Let $z \in \mathbb{Q}$, $0<z<1$, and let $z=\frac{a}{b}$  for some $a,b \in \mathbb{N}$. Let $v = b^m f_0\dots f_m$, where $f_0,\dots,f_m \in \mathbb{N}$ and $m \geq 2$. Then there exists a non-disjoint ($v, m, vz, vz^2$)-PSEDF in $\mathbb{Z}_{v}$.
        \item[(ii)] Let $v = b^mf$, where $f \in \mathbb{N}$ and $b,m \in \mathbb{Z}_{\geq 2}$. Let $1 \leq a < b$, $a, b \in \mathbb{N}$. Then there exists a non-disjoint ($b^mf, m, ab^{m-1}f, a^2b^{m-2}f$)-PSEDF in $\mathbb{Z}_{b^mf}$.
    \end{itemize}
\end{corollary}
\begin{proof}
\begin{itemize}
    \item[(i)] Let $v = b^m f_0\dots f_m$ and $z = \frac{a}{b}$ as above. We define the sequences $(d_1,\dots,d_m)=(a f_0,a f_1,\dots, a f_{m-1})$ and $(c_0,\dots,c_m) = (b f_0, b f_1, \dots, b f_{m-1}, f_m)$. We see that $\frac{d_i}{c_{i-1}} = \frac{a f_{i-1}}{b f_{i-1}}=\frac{a}{b}=z$ for all $1 \leq i \leq m$. So Theorem \ref{generalPSEDF} implies the existence of a non-disjoint ($c_0\dots c_m, m,c_0\dots c_mz,c_0\dots c_mz^2$)-PSEDF. We note that $c_0\dots c_m=b^m f_0\dots f_m=v$, so a non-disjoint ($v, m,vz,vz^2$)-PSEDF exists in $\mathbb{Z}_v$ as required.
    \item[(ii)] $f=f \cdot 1^m$ so we can take $f_0=f$, $f_1,\dots,f_m=1$ and $z = \frac{a}{b}$ and apply part (i).
\end{itemize}
\end{proof}

\begin{corollary}\label{cor:powers}
\begin{itemize}
\item[(i)] Let $S>1$, $N>2$. There exists a non-disjoint ($S^N$, $N$, $S^{N-1}$, $S^{N-2}$)-PSEDF in $\mathbb{Z}_{S^N}$ and hence a non-disjoint ($S^N,N,S^{N-1},(N-1)S^{N-2}$)-SEDF.
\item[(ii)] Let $x,y \in \mathbb{N} $, $x < y$, then there exists a non-disjoint ($y^2, 2, xy, x^2$)-PSEDF in $\mathbb{Z}_{y^2}$ and hence a a non-disjoint ($y^2, 2, xy, x^2$)-SEDF.
\end{itemize}
\end{corollary}
\begin{proof}
\begin{itemize}
     \item[(i)] Take $b=S$, $m=N$, $a=1$ and $f=1$ in Corollary \ref{generalPSEDFsparameterfriendly}(ii).
    \item[(ii)] Take $b=y$, $m=2$, $a=x$ and $f=1$ in Corollary \ref{generalPSEDFsparameterfriendly}(ii) (note $a<b$).
\end{itemize}
\end{proof}

Next, we present a new 2-set construction for non-disjoint PSEDFs and GPSEDFs, which produces some examples which are not obtainable from our previous constructions. The new construction uses a `modular arithmetic' viewpoint, and involves less repetition of blocks in the sequences.

\begin{theorem}\label{modularGPSEDF2}
Let $v=ab$ for some $a,b \in\mathbb{N}$. For any $1 \leq k_1 \leq b-1$ and $1 \leq k_2 \leq a-1$, there exists a non-disjoint $(v,2,k_1a, k_2b)$-GPSEDF in $\mathbb{Z}_v$.
\end{theorem}
\begin{proof}
Let $v=ab$ for some $a,b \in \mathbb{N}$ and fix $1 \leq k_1 \leq b-1$, $1 \leq k_2 \leq a-1$. Let $R=\{ r_{i,j} : r_{i,j} \in \mathbb{Z}, 0 \leq r_{i,j} \leq a-1 \, (0 \leq i \leq b-1, 0 \leq j \leq k_2-1)\}$ be a set with the property that $j_1 \neq j_2 \Rightarrow r_{i,j_1} \neq r_{i,j_2}$. Let
$S=\{s_i: 0 \leq s_i \leq b-1 \, (0 \leq i \leq k_1-1)\}$ be a set of $k_1$ distinct integers.\\

We next define two binary sequences of length $v$.  Let $X = (x_t)_{t=0}^{v-1}$ and $Y=(y_t)_{t=0}^{v-1}$, where:
\[x_t = \begin{cases}
	1,& jb+s_i \quad (0 \leq j \leq a-1, 0 \leq i \leq k_1-1)\\
	0,& \text{otherwise}
	\end{cases}\]
	
\[y_t = \begin{cases}
	1,& r_{i,j}b+i \quad (0 \leq i \leq b-1, 0\leq j\leq k_2-1)\\
	0,& \text{otherwise}
	\end{cases}\]

To see that the values $r_{i,j}b+i$ (which lie in the range from $0$ to $ab-1$) are distinct: there is no repetition for fixed $i$ by choice of $r_{i,j}$, and no repetition between distinct $i$ as each fixed $i$ corresponds to a different equivalence class mod $b$.

Denote by $A$ the $a k_1$-subset of $\mathbb{Z}_v$ corresponding to $X$, and by $B$ the $bk_2$-subset of $\mathbb{Z}_v$ corresponding to $Y$.  Here $A$ comprises all the elements in $\mathbb{Z}_v$, which are in the congruence classes of $s_0,\dots,s_{k_1-1}$ modulo $b$, and $B$ comprises $k_2$ elements in each congruence class modulo $b$.  
	
Let $\delta \in \mathbb{Z}_v$.  By Theorem \ref{BinSeqSets}, the number of occurrences of $\delta$ in $\Delta(A,B)$ ($=\Delta(B,A)$) equals:\\
\begin{align*}
\sum_{t=0}^{v-1} x_ty_{t+\delta}&=\sum_{i=0}^{k_1-1} \sum_{j=0}^{a-1} y_{jb+s_i+\delta}\\&\text{(by the definition of $X$, since it is 0 for all other values)}\\
&=\sum_{i=0}^{k_1-1} k_2\\
&\text{(since there are $k_2$ elements of $B$ congruent to $s_i+\delta$ modulo $b$)}\\
&=k_1k_2
\end{align*}
This holds for all $\delta \in \mathbb{Z}_v$ and so $\Delta(A,B)=\Delta(B,A)=k_1k_2 \mathbb{Z}_v$. So $\{A,B\}$ forms a non-disjoint $(v,2,k_1a,k_2b)$-GPSEDF as required.
\end{proof} 

\begin{corollary}
Let $v=ab$ for some $a,b \in\mathbb{N}$. 
Let $k_1 \in \{1, \ldots, b-1\}$ and $k_2 \in \{1, \ldots, a-1\}$.
Then there exist at least ${a \choose k_2}^b {b \choose k_1}$ different families of sets which are non-disjoint $(v,2,k_1a, k_2b)$-GPSEDFs in $\mathbb{Z}_v$.
\end{corollary}
\begin{proof}
We analyse the number of choices in the construction of Theorem \ref{modularGPSEDF2}.  For the elements of $R$ there are ${a \choose k_2}^b$ choices, then ${b \choose k_1}$ choices for the elements of $S$.  
Now we consider the sets $A$ and $B$ defined by the sequences $X$ and $Y$.

Firstly, note if $\{s_0,\dots,s_{k_1-1}\} \neq \{t_0,\dots,t_{k_1-1}\}$ then $\{jb+s_0,\dots,jb+s_{k_1-1}: 1 \leq j \leq a-1\} \neq \{jb+t_0,\dots,jb+t_{k_1-1}: 1 \leq j \leq a-1\}$ so each of the $b \choose k_1$ choices for $S$ will give us a different $X$. Similarly, if $\{r_{i,0},r_{i,1},\dots,r_{i,k_2-1}\} \neq \{t_{i,0},t_{i,1},\dots,t_{i,k_2-1}\}$ then $\{r_{i,0}b+i,r_{i,1}b+i,\dots,r_{i,k_2-1}b+i\} \neq \{t_{i,0}b+i,t_{i,1}b+i,\dots,t_{i,k_2-1}b+i\}$ and so we obtain a different $Y$ for each of the ${a \choose k_1}^b$ choices of $R$.
\end{proof}

We illustrate this construction with an example.
 \begin{example}\label{ex:fullcycorder}
     We let the parameters from Theorem \ref{modularGPSEDF2} be as follows: $v=20=4*5$, $a=4$, $b=5$, $k_1=2$, $k_2=3$, $s_0=1$, $s_1=3$ and we define $r_{i,j}$ and the values of $y_t$ according to the following tables respectively:
     \begin{center}
     \begin{tabular}{ c| c | c | c }
    \diagbox{i}{j} & $0$ & $1$ & $2$ \\
    \hline
    $0$ & $0$ & $1$ & $3$ \\ 
    \hline
    $1$ & $0$ & $2$ & $3$ \\  
    \hline
    $2$ & $0$ & $1$ & $2$ \\  
    \hline
    $3$ & $1$ & $2$ & $3$ \\  
    \hline
    $4$ & $0$ & $2$ & $3$ 
    \end{tabular}
    \quad 
    \begin{tabular}{ c| c | c | c }
    \diagbox{i}{j} & $0$ & $1$ & $2$ \\
    \hline
    $0$ & $0$ & $5$ & $15$ \\ 
    \hline
    $1$ & $1$ & $11$ & $16$ \\  
    \hline
    $2$ & $2$ & $7$ & $12$ \\  
    \hline
    $3$ & $8$ & $13$ & $18$ \\  
    \hline
    $4$ & $4$ & $14$ & $19$
    \end{tabular}
    \end{center}
We obtain the sequences:
     \begin{center}
         $X=01010010100101001010$\\
         $Y=11101101100111111011$
     \end{center}
     These sequences correspond to the sets:
     \begin{center}
         $A=\{1,3,6,8,11,13,16,18\}$\\
         $B=\{0,1,2,4,5,7,8,11,12,13,14,15,16,18,19\}$
     \end{center}
 \end{example}

We next present constructions for more than two sets which use the modular arithmetic viewpoint.

 \begin{theorem} \label{ModCoprime}
Let $v \in \mathbb{N}$ and let $a_1,\dots,a_m$ be coprime factors of $v$. 
\begin{itemize}
\item[(i)] There exists a non-disjoint $(v,m,\frac{v}{a_1},\dots,\frac{v}{a_m})$-GPSEDF in $\mathbb{Z}_v$.
\item[(ii)] Let $0 \leq \mu_i \leq a_i-1$ for $1 \leq i \leq m$. Then there exists a non-disjoint $(v,m,\mu_1 \frac{v}{a_1},\dots,\mu_m \frac{v}{a_m})$-GPSEDF in $\mathbb{Z}_v$.
\end{itemize}
\end{theorem}
\begin{proof}
(i) We let $a_1,\dots,a_m$ be coprime factors of $v$. Then we define sequences $X_i=(x_t)_{t=0}^{v-1}$ where:
    \[x_t = \begin{cases}
	1,& t=la_i\quad  (0 \leq l \leq \frac{v}{a_i}-1)\\
	0,& \text{otherwise}
	\end{cases}\]
 We let $X_i$ correspond to the set $A_i$ in $\mathbb{Z}_v$.
 Then let $1 \leq i,j \leq m$, such that $i \neq j$. We let $X_i=(x_t)_{t=0}^{v-1}$ and $X_j=(y_t)_{t=0}^{v-1}$. Let $\delta \in \mathbb{Z}_v$ and consider the number of occurrences of $\delta$ in $\Delta(A_i,A_j)$($=\Delta(A_j,A_i)$), which by Theorem \ref{BinSeqSets} is equal to:
 \begin{align*}
     \sum_{t=0}^{v-1} x_ty_{t+\delta} &= \frac{v}{a_ia_j}\sum_{t=0}^{a_ia_j-1} x_ty_{t+\delta} \quad &(X_i \text{ and } X_j \text{ repeat every } a_i \text{ or } a_j)\\
     &=\frac{v}{a_ia_j} \sum_{l=0}^{a_j-1} y_{la_i+\delta} \quad &(\text{by } X_i \text{'s definition, it is 0 for all other values})\\
     &=\frac{v}{a_ia_j} \sum_{l=0}^{a_j-1} y_{la_i+ba_j+d} \quad &(\delta = ba_j+d \text{ for some } 0 \leq b \leq \frac{v}{a_j}-1, 0 \leq d \leq a_j-1)\\
     &=\frac{v}{a_ia_j} \sum_{l=0}^{a_j-1} y_{la_i+d} \quad &(X_j \text{ repeats every } a_j)\\
     &=\frac{v}{a_ia_j}
 \end{align*}
To see that sum on the penultimate line equals 1, note that the set of indices $\{l a_i+d: 0 \leq l \leq a_i\}$ comprises all elements between $0$ and $a_i a_j-1$ which are congruent to $d$ mod $a_i$, then by the definition of $X_j$ the sum counts those set members congruent to $0$ mod $a_j$.  This sum will be 1 because $a_i$ and $a_j$ are coprime and so by the Chinese Remainder Theorem there is precisely one value between 0 and $a_ia_j-1$, which is $d$ mod $a_i$ and 0 mod $a_j$. 

This holds for all $\delta \in \mathbb{Z}_v$ and $1 \leq i, j \leq m, i \neq j$. So $\Delta(A_i,A_j)=\frac{v}{a_ia_j}$ for all $1 \leq i,j \leq m, i \neq j$. Hence, $\{A_1,\dots,A_m\}$ forms a non-disjoint $(v,m,\frac{v}{a_1},\dots,\frac{v}{a_m})$-GPSEDF in $\mathbb{Z}_v$.\\
(ii) From (i), there exists a non-disjoint $(v,m,\frac{v}{a_1},\dots,\frac{v}{a_m})$-GPSEDF $A=\{A_1,\ldots, A_m\}$ in $Z_v$, where $A_i=\{0,a_i,\ldots, (\frac{v}{a_i}-1)a_i \}$.   For each $A_i$, any choice of $\mu_i$ elements $g_j (1 \leq j \leq \mu_i)$ from $\{0,1,\ldots, a_i-1\}$ will yield $\mu_i$ sets $g_j+A_i$ which are disjoint.  By Theorem \ref{thm:combinations} (v), replacing each $A_i$ in $A$ with such a union of $g_j+A_i$, yields a non-disjoint GPSEDF with the stated parameters.
\end{proof}

\begin{example}
Let $v=15$ and let $a_1=5$ and $a_2=3$.  The sets $A_1=\{0,5,10\}$ and $A_1=\{0,3,6,9,12\}$ form a non-disjoint $(15,2,3,5)$-GPSEDF in $\mathbb{Z}_{15}$.  \\
Let $\mu_1=3$ and $\mu_2=2$; taking the unions of $\{A_1,1+A_1,2+A_1\}$ and $\{A_2, 1+A_2\}$ we obtain
$$\{ \{0,1,2,5,6,7,10,11,12\} ,\{0,1,3,4,6,7,9,10,12,13\}\}$$ 
which is a non-disjoint $(15,2,9,10)$-GPSEDF in $\mathbb{Z}_{15}$.
\end{example}

We end this section with a multiset version of Theorem \ref{GPSEDFblock}.
\begin{theorem}
Let $(a_0, a_1,\dots,a_m)$ ($m \geq 2$) be a sequence of positive integers such that for all $1 \leq i+1 \leq m$, $a_i = b_{i+1}a_{i+1}$ for some $b_{i+1} (\geq 2) \in \mathbb{Z}$. Then, for all possible $1 \leq k_i \leq b_i$ there exists a non-disjoint ($a_0,m,(l_{1,1}+l_{1,2}+\dots+l_{1,k_1})a_1,\dots,(l_{m,1}+l_{m,2}+\dots+l_{m,k_m})a_m\frac{a_0}{a_{m-1}}$)-MGPSEDF, where $l_{i,j}$ is an arbitrary natural number for all $1 \leq i \leq m, 1 \leq j \leq k_m$. Its $\lambda$-matrix is $[a_{ij}]_{m \times m}$, where $a_{ij} = 0$ if $i=j$ and $a_{ij}=\lambda_{i,j}=\frac{a_0(l_{i,1}+l_{i,2}+\dots+l_{i,k_i})a_i(l_{j,1}+l_{j,2}+\dots+l_{j,k_j})a_j}{a_{i-1}a_{j-1}}$ otherwise.
\end{theorem}
\begin{proof}
Let $(a_0, a_1,\dots,a_m)$ ($m \geq 2$) be a sequence of positive integers such that for all $1 \leq i+1 \leq m$, $a_i = b_{i+1}a_{i+1}$ for some $b_{i+1} \geq 2$. Then fix $1 \leq k_i \leq b_i$ for all $1 \leq i \leq m$.\\

Define $X_i=(x_t)_{t=0}^{a_0-1}$ for all $i = 1,\dots,m$, where:
\[x_t =\begin{cases}
    l_{i,j},& (j-1)a_i\leq t \leq ja_i-1, j = 1,\dots,k_i\\
    0,& k_ia_i\leq t \leq a_ib_i-1=a_{i-1}-1\\
    x_{t-a_ib_i}, & a_ib_i \leq t \leq a_0 -1\quad (a_ib_i=a_{i-1})
    \end{cases}\]

The proof is analogous to that of Theorem \ref{GPSEDFblock}, using Theorem \ref{thm:multi_seq}.
\end{proof}
\end{section}

\begin{section} {New results on non-disjoint SEDFs and GSEDFs using group theory}\label{section:groups}

In this section we show how adopting a group theoretic viewpoint can be a productive way to yield further constructions in $\mathbb{Z}_n$ and other abelian and non-abelian groups. We write groups multiplicatively except when working in specific additive groups.  We require the following group theoretic results.

\begin{lemma}\label{lem:GroupTheory}
Let $G$ be a finite group and let $H,K \leq G$.  Define the map  $f: H \times K \rightarrow HK$ by $(h,k) \mapsto hk$
\begin{itemize}
\item[(i)] $HK$ is a subgroup of $G$ if and only if $HK=KH$.
\item[(ii)] Under $f$, the preimage of any element of $HK$ has cardinality $|H \cap K|$.
\item[(iii)] $|HK|=\frac{|H| |K|}{|H \cap K|}$;
\item[(iv)] $G=HK \Leftrightarrow |G|=|HK| 
\Leftrightarrow [G: H \cap K]=[G:H] [G:K]$.
\end{itemize}
\end{lemma}
\begin{proof}
We prove (ii) and (iii).  Define the map $f: H \times K \rightarrow HK$ by $(h,k) \mapsto hk$.  Clearly this map is surjective.   Let $x\in HK$; say $x=hk$ for some fixed $h \in H, k \in K$.  For $h' \in H, k' \in K$ we have $f(h',k')=x \Leftrightarrow h'k'=hk \Leftrightarrow h^{-1}h'=kk'^{-1} \in H \cap K \Leftrightarrow$ for some $y \in H \cap K$, $h'=hy$ and $k'=y^{-1}k$.  Hence the preimage under $f$ of $hk$ is the set $\{(hy,y^{-1}k): y \in H \cap K\}$ which is clearly in bijective correspondence with $H \cap K$.  So $|H \times K|=|HK||H \cap K|$ and so $|H||K|=|HK||H \cap K|$. 
\end{proof}

\begin{lemma}\label{cor:coprime}
Let $G$ be a finite group and let $H,K \leq G$. If $[G:H]$ and $[G:K]$ are coprime integers then $G=HK$.
\end{lemma}

We next present a non-disjoint GPSEDF construction which utilises these group theoretic properties.

\begin{theorem} \label{thm:groups}
Let $G$ be a finite group and let $H=\{H_1, \ldots, H_m\}$ be a collection of distinct subgroups of $G$.  If $G=H_i H_j$ for all pairs $i \neq j$ ($1 \leq i,j \leq m$) then
\begin{itemize}
\item[(i)] $H$ is a non-disjoint $(|G|,m,|H_1|, \ldots, |H_m|)$-GPSEDF with $\lambda_{i,j}=|H_i \cap H_j|$ for all $i \neq j$.
\item[(ii)] Any $H_i \in H$ can be replaced in $H$ by a union of $\mu_i$ distinct cosets of $H_i$ to obtain a new  non-disjoint $(|G|,m,|H_1|, \ldots, \mu_i|H_i|, \ldots, |H_m|)$-GPSEDF.
\end{itemize}
\end{theorem}
\begin{proof}
We first prove (i).  Let $i \neq j$.  The multiset $\Delta(H_i,H_j)=\{xy^{-1}: x \in H_i, y \in H_j\}=\{xy: x \in H_i, y \in H_j\}$ since $H_j$ is a subgroup.  By Lemma \ref{lem:GroupTheory},  for each $g \in H_i H_j=G$, there are $|H_i \cap H_j|$ pairs $(x,y) \in H_i \times H_j$ such that $xy=g$.  Hence $\Delta(H_i,H_j)$ comprises $|H_i \cap H_j|$ occurrences of each element of $G$.  Part (ii) follows by applying Theorem \ref{thm:combinations} (v).
\end{proof}

The following result presents some specific constructions obtainable from Theorem \ref{thm:groups}.  Since there is no requirement for the group $G$ to be abelian, this yields our first examples of non-abelian non-disjoint SEDFs, PSEDFs and GPSEDFs.
\begin{corollary}\label{cor:groups}
\begin{itemize}
\item[(i)] Let $G=\mathbb{Z}_p \times \mathbb{Z}_p$, where $p$ is prime. Let $H$ be the collection of all $p+1$ distinct subgroups of $G$ of order $p$.  Then $H$ is a non-disjoint $(p^2,p+1,p,1)$-PSEDF.
\item[(ii)] Let $G=D_{2n}$ (the dihedral group of order $2n$) with $n$ odd.  Let $\alpha$ and $\beta$ be the standard generators for $D_{2n}$ such that  $|\alpha|=n, |\beta|=2$.  Then $\{\langle \alpha \rangle, \langle \beta \rangle \}$ is a non-disjoint $(2n,2,2,n)$-GPSEDF.
\item[(iii)] Let $G=\mathbb{Z}_n$.  Let $\{a_1, \ldots, a_m\}$ be a collection of divisors of $n$ which are pairwise coprime.   For $1 \leq i \leq m$,  let $H_i$ be the additive subgroup generated by $a_i$.  Then $H=\{H_1,\ldots,H_m\}$ is a non-disjoint $(n,m,\frac{n}{a_1},\ldots,\frac{n}{a_m})$-GPSEDF.
\end{itemize}
\end{corollary}

Observe that the construction of Corollary \ref{cor:groups}(iii) corresponds to Proposition \ref{ModCoprime} (i), while its extension via Theorem \ref{thm:groups} (ii) corresponds to 
Proposition \ref{ModCoprime} (ii).

There are many other constructions in both abelian and non-abelian groups which result from an application of Theorem \ref{thm:groups}.  We present the following examples.

\begin{example}
Let $G=D_8$ be the dihedral group of order 8, which is generated by $\alpha, \beta$ where $|\alpha|=4, |\beta| =2$. Then $H=\{e, \beta, \alpha^2, \alpha^2 \beta\}$ and $K=\{e, \alpha\beta, \alpha^2, \alpha^3 \beta \}$ are subgroups of $D_8$, such that $HK=KH=D_8$ and $H \cap K=\{e,\alpha^2\}$. By Theorem \ref{thm:groups}, $\{H,K\}$ is a non-disjoint $(8,2,4,4)$-GPSEDF; the following subtraction table illustrates that $\Delta(H,K)=2D_8$

\begin{center}
\renewcommand\arraystretch{1.3}
\setlength\doublerulesep{0pt}
\begin{tabular}{ c| c | c | c | c }
$*$ & $e$ & $\alpha \beta$ & $\alpha^2$ & $\alpha^3 \beta$ \\
\hline\hline
$e$ & $e$ & $\alpha \beta$ & $\alpha^2$ & $\alpha^3 \beta$ \\ 
\hline
$\beta$ & $\beta$ & $\alpha^3$ & $\alpha^2 \beta$ & $\alpha$ \\ 
\hline
$\alpha^2$ & $\alpha^2$ & $\alpha^3 \beta$ & $e$ & $\alpha \beta$ \\ 
\hline
$\alpha^2 \beta$ & $\alpha^2 \beta$ & $\alpha$ & $\beta$ & $\alpha^3$
\end{tabular}
\end{center}
\end{example}

\begin{example}
Let $G=Q_8$ be the quaternion group, with its elements written in usual notation as $\{1,-1,i,-i,j,-j,k,-k\}$.  The subgroups $\langle i \rangle$ and $\langle j \rangle$ in $Q_8$ form a non-disjoint ($8,2,4,2$)-PSEDF. The subgroups $\langle i \rangle, \langle j \rangle$ and $\langle k \rangle$ form a non-disjoint ($8,3,4,2$)-PSEDF.  The following table illustrates $\Delta(\langle i \rangle, \langle j \rangle)$.
\begin{center}
\renewcommand\arraystretch{1.3}
\setlength\doublerulesep{0pt}
\begin{tabular}{ c| c | c | c | c }
$*^{-1}$ & $1$ & $j$ & $-j$ & $-1$ \\
\hline\hline
$1$ & $1$ & $-j$ & $j$ & $-1$ \\ 
\hline
$i$ & $i$ & $-k$ & $k$ & $-i$ \\ 
\hline
$-i$ & $-i$ & $k$ & $-k$ & $i$ \\ 
\hline
$-1$ & $-1$ & $-j$ & $j$ & $1$
\end{tabular}
\end{center}
\end{example}

In fact, we can develop a more flexible construction in $\mathbb{Z}_n$ than that of Corollary \ref{cor:groups}(iii) - one which does not require the individual subgroups to have coprime index.  In $\mathbb{Z}_n$, we return to additive notation.  For $a|n$, we will denote the additive subgroup of $(\mathbb{Z}_n,+)$ generated by $a \in \mathbb{Z}_n$ (the set $\{0,a,\ldots,n-a\}$) by $a \mathbb{Z}_n$, to make it clear in which group the cyclic subgroup lies. Here we avoid the use of $\lambda S$ for $\lambda$ copies of set $S$, which we have used in other sections.

We will need the following preliminary result.

\begin{lemma}\label{lem:gcd}
Let $a,b$ be distinct divisors of $n$ ($a,b,n \in \mathbb{N}$). Let $d=gcd(a,b)$ and $l=lcm(a,b)$.  Then
\begin{itemize}
\item[(i)] The multiset $\Delta(a \mathbb{Z}_n, b \mathbb{Z}_n)$ comprises $\frac{n}{l}$ copies of $d\mathbb{Z}_n$;
\item[(ii)] If $d>1$ then 
$$ \Delta(\cup_{r=0}^{d-1} (r+ a \mathbb{Z}_n), b \mathbb{Z}_n)= \Delta(a \mathbb{Z}_n, \cup_{r=0}^{d-1} (r+ b \mathbb{Z}_n))=\frac{n}{l} \mbox{ copies of } \mathbb{Z}_n$$ 
i.e. $\{ \{ \cup_{r=0}^{d-1} (r+ a \mathbb{Z}_n)\}, b \mathbb{Z}_n\}$ and $\{ a \mathbb{Z}_n, \{\cup_{r=0}^{d-1} (r+ b \mathbb{Z}_n)\} \}$ both form 2-set non-disjoint GPSEDFs with $\lambda=\frac{n}{l}$.
\end{itemize}
\end{lemma}
\begin{proof}
For (i), by Lemma \ref{lem:GroupTheory}, $a \mathbb{Z}_n+b \mathbb{Z}_n$ is the subgroup of $\mathbb{Z}_n$ with order $\frac{|a\mathbb{Z}_n||b \mathbb{Z}_n|}{|a\mathbb{Z}_n \cap b \mathbb{Z}_n|}=\frac{(n/a)(n/b)}{n/l}=\frac{n}{d}$, i.e. $d \mathbb{Z}_n$, and in the internal difference multiset $\Delta(a\mathbb{Z}_n,b\mathbb{Z}_n)$, each element of $d \mathbb{Z}_n$ occurs $\frac{n}{l}$ times.  Part (ii) then follows.
\end{proof}

We first present the most general version of our non-disjoint GPSEDF construction, using an algorithm which allows considerable flexibility.

\begin{definition}\label{def:alg}
Let $n \in \mathbb{N}$ and let $a_1, \dots, a_m \in \mathbb{N}$ be distinct divisors of $n$.
We define $chunk(a_i)$ to be the final value after the following algorithm has been run:
    \begin{enumerate}
        \item Set $chunk(a_i) = 1$ for $1 \leq i \leq m$
         %For all $1 \leq i < j \leq m$
        \item For all $i$ from $1$ to $m-1$
        \item For all $j$ from $i+1$ to $m$
        \begin{enumerate}
            \item Let $d = gcd(a_i,a_j)$
            \item Arbitrarily choose $i$ or $j$; call this $k$.
            \item Redefine $chunk(a_k) := lcm(chunk(a_k),d)$
        \end{enumerate}
    \end{enumerate}
\end{definition}

\begin{theorem}\label{thm:AlgGps}
Let $G=\mathbb{Z}_n$.  Let $\{a_1,\ldots,a_m \}$ be a collection of distinct divisors of $n$.  Let $ch(a_i)$ be the value of $chunk(a_i)$ as defined in Definition \ref{def:alg} ($1 \leq i \leq m$).  Define 
$$H_{i}=\cup_{r=0}^{ch(a_i)-1} (r+ a_{i}\mathbb{Z}_n).$$  
Then $H=\{H_1, \ldots, H_m \}$ is a non-disjoint $(n, m, \frac{n}{a_1} ch(a_1), \frac{n}{a_2} ch(a_2),\ldots, \frac{n}{a_m} ch(a_m))$-GPSEDF in $\mathbb{Z}_n$ with $\lambda_{i,j}=\frac{n}{a_i a_j} ch(a_i) ch(a_j) $.
\end{theorem}
\begin{proof}
Let $1 \leq i < j \leq m$.  We aim to show that $\{H_i,H_j\}$ is a non-disjoint GPSEDF with the stated parameters.  We may replace precisely one of $H_i$ or $H_j$ by $a_i \mathbb{Z}_n$ or $a_j \mathbb{Z}_n$ respectively, then apply Theorem \ref{thm:combinations}(v) to complete the result.

Let $l=lcm(a_i,a_j)$.  Then $a_i \mathbb{Z}_n$ and $a_j \mathbb{Z}_n$ are unions of cosets of $l \mathbb{Z}_n$; specifically $a_i \mathbb{Z}_n= l \mathbb{Z}_n \cup (a_i+ l \mathbb{Z}_n) \cup \cdots \cup ((\frac{l}{a_i}-1)a_i + l \mathbb{Z}_n)$ and $a_j \mathbb{Z}_n= l \mathbb{Z}_n \cup (a_j+ l \mathbb{Z}_n) \cup \cdots \cup ((\frac{l}{a_j}-1)a_j + l \mathbb{Z}_n)$.

At some point in the algorithm of Definition \ref{def:alg}, the pair $i<j$ was considered, one of $\{chunk(a_i),chunk(a_j)\}$ was arbitrarily chosen, and assigned a new value divisible by $\gcd(a_i,a_j)$.  We may assume wlog that this was $chunk(a_i)$ (if it was $chunk(a_j)$ then the following argument is exactly analogous with $i$ and $j$ switched).  From the algorithm, this divisibility property persists irrespective of any subsequent updates to the value of $chunk(a_i)$, so we have that $\gcd(a_i,a_j)|ch(a_i)$.  We now consider 
$$\Delta(H_i, a_j \mathbb{Z}_n)=\Delta( \cup_{r=0}^{ch(a_i)-1} \cup_{s=0}^{\frac{l}{a_i}-1} (r+sa_i+l \mathbb{Z}_n), \cup_{t=0}^{\frac{l}{a_j}-1} t a_j+ l \mathbb{Z}_n)$$

By Theorem \ref{thm:cosets}, this equals $\lambda$ copies of $\mathbb{Z}_n$ for some $\lambda$ if the sets $\{ r+sa_i: 0 \leq r \leq ch(a_i)-1, 0 \leq s \leq \frac{l}{a_i}-1\}$  and $\{t a_j: 0 \leq t \leq \frac{l}{a_j}-1 \}$ form a non-disjoint GPSEDF in $G/l \mathbb{Z}_n \cong \mathbb{Z}_l$.

The relevant difference multiset in $\mathbb{Z}_l$ is $$\Delta(a_i \mathbb{Z}_l \cup (1+ a_i \mathbb{Z}_l) \cup \cdots \cup ((ch(a_i)-1)+a_i \mathbb{Z}_l), a_j \mathbb{Z}_l)$$ which is $\cup_{r=0}^{ch(a_i)-1} \Delta(r+ a_i \mathbb{Z}_l, a_j \mathbb{Z}_l)$.  By Lemma \ref{lem:gcd}, the elements of $\Delta(a_i \mathbb{Z}_l, a_j \mathbb{Z}_l)$ are those of $\gcd(a_i,a_j) \mathbb{Z}_l$, each occurring $\frac{l}{lcm(a_i,a_j)}=1$ times.  Hence the elements of $\cup_{r=0}^{ch(a_i)-1} \Delta(r+ a_i \mathbb{Z}_l, a_j \mathbb{Z}_l,)$ are the elements of $\cup_{r=0}^{ch(a_i)-1} (r+ gcd(a_i,a_j) \mathbb{Z}_l)$, i.e. $\frac{ch(a_i)}{\gcd(a_i,a_j)}$ copies of the multiset $\cup_{r=0}^{gcd(a_i,a_j)-1} (r+ gcd(a_i,a_j) \mathbb{Z}_l)=\mathbb{Z}_l$. So we have a non-disjoint $(l,2,\frac{ch(a_i)l}{a_i},\frac{l}{a_j})$-GPSEDF in $\mathbb{Z}_l$ with $\lambda_{1,2}=\frac{ch(a_i)}{\gcd(a_i,a_j)}$ . Hence in $G$ we have a non-disjoint $(n,2,\frac{ch(a_i)n}{a_i},\frac{n}{a_j})$-GPSEDF with $\lambda=\frac{ch(a_i)n}{\gcd(a_i,a_j)l}=\frac{ch(a_i)n}{a_i a_j}$.  Finally, $\{H_i,H_j\}$ form a non-disjoint $(n,2,\frac{ch(a_i)n}{a_i},\frac{ch(a_j) n}{a_j})$-GPSEDF in $G$ with $\lambda_{i,j}=\frac{ch(a_i) ch(a_j) n}{a_i a_j}$ and so the result follows.
\end{proof}

From Theorem \ref{thm:AlgGps}, we can obtain various corollaries, whose statements give explicit parameters.  For example, we have the following three-set construction.

\begin{corollary}\label{thm:3sets}
Let $G=\mathbb{Z}_n$.  Let $\{a_1,a_2,a_3\}$ be a collection of distinct divisors of $n$.  Taking the indices on the $a$'s modulo $3$, define (for $1 \leq i \leq 3$):
$$H_{i}=\cup_{r=0}^{\gcd(a_i,a_{i+1})-1} (r+ a_{i}\mathbb{Z}_n).$$  
Then $H=\{H_1, H_2,H_3\}$ is a non-disjoint $(n, 3, \frac{n}{a_1} gcd(a_1,a_2), \frac{n}{a_2} gcd(a_2,a_3), \frac{n}{a_3} gcd(a_3,a_1))$-GPSEDF in $\mathbb{Z}_n$ with $\lambda_{i,i+1}=\lambda_{i+1,i}=\frac{n \gcd(a_{i+1},a_{i+2})}{lcm(a_i,a_{i+1})}$. 
\end{corollary}
\begin{proof}
Apply the algorithm of Definition \ref{def:alg} as follows: for $i=1$ and $j=2$ take $k=1$; for $i=1$ and $j=3$ take $k=3$; and for $i=2$ and $k=3$ take $k=2$. So $ch(a_1)=\gcd(a_1,a_2)$, $ch(a_2)=\gcd(a_2,a_3)$ and $ch(a_3)=\gcd(a_3,a_1)$. Now apply Theorem \ref{thm:AlgGps}.
\end{proof}

\begin{example}\label{ex:3sets}
Let $G=\mathbb{Z}_{24}$ and let $a_1=4$, $a_2=6$ and $a_3=8$.  From Corollary \ref{thm:3sets} we obtain the non-disjoint GPSEDF with sets
$\{0,1,4,5,8,9,12,13,16,17,20,21\}$, $\{0,1,6,7,12,13,18,19\}$, and $\{0,1,2,3,8,9,10,11,16,17,18,19\}$
where $\lambda_{1,2}=4$, $\lambda_{2,3}=4$ and $\lambda_{3,1}=6$  These correspond to the sequences:
\begin{itemize}
\item[]$110011001100110011001100$
\item[]$110000110000110000110000$
\item[]$111100001111000011110000$
\end{itemize}
\end{example}

In the final section of this paper, we show how an adaptation of this approach will yield a construction for VW-OOCs with minimum cross-correlation, which do not possess the GPSEDF property.

Finally, the following general result demonstrates a connection to classical EDFs and GEDFs (though not the strong versions). Recall that a family of subgroups of a group $G$ is called a partition of $G$ if every non-identity element of $G$ belongs to precisely one subgroup in the family \cite {Zap}.
\begin{theorem}
Let $G$ be a finite group and let $H=\{H_1,\ldots,H_m\}$ be a collection of subgroups of $G$ such that $|G|=|H_i||H_j|$ for all $i \neq j$ and $H$ is a partition of $G$.  Then 
\begin{itemize}
\item[(i)]$H$ is a non-disjoint GPSEDF with $\lambda_{i,j}=1$ for all $i,j$ (a non-disjoint PSEDF when all $|H_i|$ are equal);
\item[(ii)] $H'=\{H_1 \setminus \{e\}, \ldots, H_m \setminus \{e\}\}$ is a classical GEDF (classical EDF when all $|H_i|$ are equal) which is not strong.
\end{itemize}
\end{theorem}
\begin{proof}
Part (i) is immediate from Lemma \ref{lem:GroupTheory} and Theorem \ref{thm:groups}.  For (ii), removing the identity from each $H_k$ ($1 \leq k \leq m$) makes the sets in $H'$ pairwise disjoint.  From each $\Delta(H_i,H_j)$, it removes the occurrence of $\{e\}$, one copy of $H_i \setminus \{e\}$ and one copy of $H_j \setminus \{e\}$.  Overall, an equal number of copies of all sets in $H'$ are removed from the original external difference multiset of the non-disjoint GPSEDF $H$, together with all occurrences of the identity.  Since the subgroups of $H$ form a subgroup partition of $G$, the resulting external difference multiset yields each non-identity element an equal number of times, and so $H'$ is a classical GEDF.   This is not strong since, from $\cup_{j \neq i} \Delta(H_i,H_j)$ ($i$ fixed), $m-1$ copies of $H_i$ would be removed but only one copy each of $H_j$ , $j \neq i$.
\end{proof}
This can be applied to the construction in $\mathbb{Z}_p \times \mathbb{Z}_p$ in Corollary \ref{cor:groups}(i).  The classical EDF case of this result relates to the classical EDF case of a partition construction given in \cite{HucPatWEDF}.  
\end{section}

\begin{section}{Non-disjoint PSEDF and GPSEDF Constructions in External Direct Products}
In this section, we show how we can combine non-disjoint PSEDFs and non-disjoint SEDFs to create new ones in external direct products. This is in contrast to classical SEDFs, in which no extensions to direct products are known. 

We first show how to combine individual non-disjoint PSEDFs to create new ones in direct products of groups.

\begin{theorem}
Let $A = \{A_1, A_2,\dots,A_m\}$ and $B= \{B_1, B_2,\dots,B_m\}$ be a non-disjoint $(n_a,m,k_a,\lambda_a)$-PSEDF and a non-disjoint $(n_b,m,k_b,\lambda_b)$-PSEDF in groups $G_{a}$ and $G_{b}$ respectively. Then $L=\{A_i\times B_i: 1 \leq i \leq m\}$ is a non-disjoint $(n_an_b,m,k_ak_b,\lambda_a \lambda_b)$-PSEDF in $G_{a}\times G_{b}$.
\end{theorem}

\begin{proof}
Let $C_i = A_i \times B_i$ ($1 \leq i \leq m$). Then $L=\{C_i: 1 \leq i \leq m\}$, with $|C_i|=|A_i||B_i|=k_ak_b$ and $C_i \subseteq G_{a} \times G_{b}$ for each $i$.\\
Now fix $(\delta_1,\delta_2) \in G_{a} \times G_{b}$ and consider the number of occurrences of $(\delta_1,\delta_2)$ in $\Delta(C_i,C_j)$ ($1 \leq i,j \leq m$, $i \neq j$). This is the number of $\{(a_i,b_i),(a_j,b_j)\} \in C_i \times C_j$ such that:
\[(\delta_1, \delta_2) = (a_i,b_i) - (a_j, b_j) =  (a_i-a_j,b_i-b_j)\]
The number of pairs of elements $a_i \in A_i$ and $a_j \in A_j$ such that $a_i-a_j=\delta_1$ is $\lambda_a$ (as $A_i, A_j$ are sets of a non-disjoint $(n_a,m,k_a,\lambda_a)$-PSEDF). Similarly, there are $\lambda_b$ possibilities for pairs of elements $b_i$ and $b_j$. So there are $\lambda_a \lambda_b$ occurrences of $(\delta_1, \delta_2) \in \Delta(C_i, C_j)$. This holds for all $(\delta_1, \delta_2) \in G_{a} \times G_{b}$ and all $1 \leq i,j \leq m, i \neq j$. The result follows.
\end{proof}

\begin{example}
In the Klein four-group $K_4=\{e,a,b,c\}$, the sets $\{ A_a=\{e,a\}, A_b=\{e,b\}, A_c=\{e,c\} \}$ form a non-disjoint $(4,3,2,1)$-PSEDF.  The sets $L=\{ A_a \times A_a, A_b \times A_b, A_c \times A_c\}$ form a non-disjoint $(16,3,4,1)$-PSEDF in $K_4 \times K_4 \cong (\mathbb{Z}_2)^4$.  The external difference table for $A_a \times A_a$ and $A_b \times A_b$ is below.
\begin{center}
\renewcommand\arraystretch{1.3}
\setlength\doublerulesep{0pt}
\begin{tabular}{ c| c | c | c | c }
$*^{-1}$ & $(e,e)$ & $(e,b)$ & $(b,e)$ & $(b,b)$ \\
\hline\hline
$(e,e)$ & $(e,e)$ & $(e,b)$ & $(b,e)$ & $(b,b)$ \\ 
\hline
$(e,a)$ & $(e,a)$ & $(e,c)$ & $(b,a)$ & $(b,c)$ \\ 
\hline
$(a,e)$ & $(a,e)$ & $(a,b)$ & $(c,e)$ & $(c,b)$ \\ 
\hline
$(a,a)$ & $(a,a)$ & $(a,c)$ & $(c,a)$ & $(c,c)$
\end{tabular}
\end{center}
\end{example}

Analogous proof techniques yield the following:
\begin{corollary}
    Let $A=\{A_1,\dots,A_m\}$ and $B=\{B_1,\dots,B_m\}$ be families of sets in groups $G_a$ and $G_b$ respectively. Let $L=\{A_i\times B_i: 1 \leq i \leq m\}$.
    \begin{itemize}
        \item [(i)] If $A$ and $B$ are a non-disjoint $(n_a,m,k_{a_1},\dots,k_{a_m})$-GPSEDF and a non-disjoint $(n_b,m,k_{b_1},\dots,k_{b_m})$-GPSEDF respectively, then $L$ is a non-disjoint $(n_an_b,m,k_{a_1}k_{b_1},\dots,k_{a_m}k_{b_m})$-GPSEDF in $G_{a}\times G_{b}$. If the $\lambda$-matrices of $A$ and $B$ are $[a_{ij}]_{m \times m}$ and $[b_{ij}]_{m \times m}$ respectively, then the $\lambda$-matrix of $L$ is $[a_{ij} b_{ij}]_{m \times m}$
        \item [(ii)] If $A$ and $B$ are a non-disjoint $(n_a,m,k_{a_1},\dots,k_{a_m})$-MGPSEDF and a non-disjoint $(n_b,m,k_{b_1},\dots,k_{b_m})$-MGPSEDF respectively, then $L$ is a non-disjoint $(n_an_b,m,k_{a_1}k_{b_1},\dots,k_{a_m}k_{b_m})$-MGPSEDF in $G_{a}\times G_{b}$. If the $\lambda$-matrices of $A$ and $B$ are $[a_{ij}]_{m \times m}$ and $[b_{ij}]_{m \times m}$ respectively, then the $\lambda$-matrix of $L$ is $[a_{ij} b_{ij}]_{m \times m}$
        \item [(iii)] If $A$ and $B$ are a non-disjoint $(n_a,m,k_{a_1},\dots,k_{a_m};\lambda_1,\dots,\lambda_m)$-GSEDF and a non-disjoint $(n_b,m,k_{b_1},\dots,k_{b_m};\mu_1,\dots,\mu_m)$-GSEDF respectively, then $L$ is non-disjoint $(n_an_b,m,k_{a_1}k_{b_1},\dots,k_{a_m}k_{b_m};\lambda_1 \mu_1,\dots,\lambda_m \mu_m)$-GSEDF.
        \item[(iv)] If $A$ and $B$ are a non-disjoint $(n_a,m,k_{a_1},\dots,k_{a_m})$-MGSEDF and non-disjoint $(n_b,m,k_{b_1},\dots,k_{b_m})$-MGSEDF respectively, then $L$ is a non-disjoint $(n_an_b,m,k_{a_1}k_{b_1},\dots,k_{a_m}k_{b_m})$-MGSEDF in $G_{a}\times G_{b}$. If the $\lambda$-matrices of $A$ and $B$ are $[a_{ij}]_{m \times m}$ and $[b_{ij}]_{m \times m}$ respectively, then the $\lambda$-matrix of $L$ is $[a_{ij} b_{ij}]_{m \times m}$
    \end{itemize}
\end{corollary}

This technique can be applied recursively to create constructions in direct products of more than two groups.
\end{section}

\begin{section}{New results on classical SEDFs and GSEDFs}

In this section we show how the sequence techniques of Section \ref{section:seq} may be used to establish results about  classical SEDFs. We note that for a classical SEDF (with sets corresponding to the sequences $(x_t)_{t=0}^{v-1}$ and $(y_t)_{t=0}^{v-1}$) we now want $\sum_{t=0}^{v-1} x_ty_{t+\delta}=\lambda$ for all $1 \leq \delta \leq v-1$ and $\sum_{t=0}^{v-1} x_ty_t=0$.

The next theorem encompasses all known classical SEDFs which are not obtainable from cyclotomy, via a single direct sequence-based construction which elucidates their structure. We note that these could be obtained recursively by taking a suitable initial classical GPSEDF and repeatedly applying the recursive construction of \cite{HucJefNep}; our result does not require prior knowledge of a suitable initial structure.

(We observe that the cyclotomic construction of a 2-set classical SEDF using the non-zero squares and non-squares of a finite field can also be proved via sequence techniques, but do not include this here.)

\begin{theorem}\label{GSEDF}
Let $h_1,h_2,h_3,h_4 \in \mathbb{N}$.  There exists a classical ($h_1h_2h_3h_4+1, 2, h_1h_3, h_2h_4,1,1$)-GSEDF in $\mathbb{Z}_{h_1h_2h_3h_4+1}$ with sets given by $$A=\cup_{i=0}^{h_3-1} \{ih_1h_2 + \alpha : 0\leq \alpha \leq h_1-1\}$$ 
and 
$$B=\cup_{i=0}^{h_4-1} \{ih_1h_2h_3+ h_1h_2(h_3-1) + \beta h_1: 1 \leq \beta \leq h_2\}$$
\end{theorem}

\begin{proof}
Let $X = (x_t)_{t=0}^{h_1h_2h_3h_4}$ and $Y = (y_t)_{t=0}^{h_1h_2h_3h_4}$, be the sequences corresponding to $A$ and $B$ respectively so:

	\[x_t =\begin{cases}
    1,& 0\leq t \leq h_1-1\\
    0,& h_1\leq t \leq h_1h_2-1\\
    x_{t-h_1h_2},& h_1h_2\leq t \leq h_1h_2h_3-1\\
    0,&  h_1h_2h_3\leq t \leq h_1h_2h_3h_4
    \end{cases}\]
    
    \[y_t =\begin{cases}
    0,& t=0\\
    0,& 1\leq t \leq h_1h_2(h_3-1)\\
    0,& h_1h_2(h_3-1)+1 \leq t \leq h_1h_2(h_3-1)+h_1-1\\
    1,& t=h_1h_2(h_3-1)+h_1\\
    y_{t-h_1},& h_1h_2(h_3-1)+h_1+1 \leq t \leq h_1h_2h_3\\
    y_{t-h_1h_2h_3},& h_1h_2h_3+1 \leq t \leq h_1h_2h_3h_4
    \end{cases}\]
    
So we obtain:

\[X = \overbrace{\overbrace{\overbrace{1\dots1}^{h_1}\overbrace{0\dots0}^{h_1}\dots\overbrace{0\dots0}^{h_1}}^{h_2 \text{ blocks of length } h_1}\dots\overbrace{\overbrace{1\dots1}^{h_1}\overbrace{0\dots0}^{h_1}\dots \overbrace{0\dots0}^{h_1}}^{h_2 \text{ blocks of length } h_1}}^{h_3 \text{ repeats of length-} h_1h_2 \text{ blocks}}\textbf{0}00\dots\dots\dots\dots0\]

\[Y = \textbf{0}\overbrace{\overbrace{00\dots\dots0\textbf{0}\underbrace{\underbrace{0\dots01}_{h_1}\underbrace{0\dots01}_{h_1}\dots\underbrace{00\dots01}_{h_1}}_{h_2 \text{ repeats}}}^{\text{length } h_1h_2h_3}\dots\overbrace{00\dots\dots0\textbf{0}\underbrace{\underbrace{0\dots01}_{h_1}\underbrace{0\dots01}_{h_1}\dots\underbrace{0\dots01}_{h_1}}_{h_2 \text{ repeats}}}^{\text{length } h_1h_2h_3}}^{h_4 \text{ repeats of length-} h_1h_2h_3 \text{ blocks}}\]

Let $A$ and $B$ be the sets corresponding to $X$ and $Y$. We note that they are both sets in $\mathbb{Z}_{h_1h_2h_3h_3+1}$ as $X$ and $Y$ both have length $h_1h_2h_3h_4+1$.

We see that $X$ comprises $h_3$ length-$h_1 h_2$ substrings (each substring containing $h_1$ $1$s), followed by a gap of length $h_1 h_2 h_3(h_4-1)+1$. So $X$ has weight $h_1h_3$.  Sequence $Y$ comprises a $0$, then $h_4$ substrings each of length $h_1h_2h_3$, each substring containing $h_2$ $1$s; hence $Y$ has weight $h_2h_4$.

$A$ and $B$ are disjoint since $X$'s rightmost 1 occurs in position $h_1h_2h_3-1 - h_1h_2 + h_1 = h_1h_2(h_3 - 1) + h_1 -1$ but $Y$'s leftmost 1 occurs in position $h_1h_2(h_3-1)+h_1$. So $ 0 \notin \Delta(A,B)$.

Let $\delta \in \{1,\dots,h_1h_2h_3h_4\}$. We consider the number of occurrences of $\delta$ in $\Delta(A,B)$, which by Theorem \ref{BinSeqSets} and the definition is equal to:

\[\sum_{t=0}^{h_1h_2h_3h_4} x_ty_{t+\delta} = \sum_{t=0}^{h_1h_2h_3-1} x_ty_{t + \delta}\]
since $x_t = 0$ for $h_1h_2h_3 \leq t \leq h_1h_2h_3h_4$.
We claim that this sum equals $1$.  

Excluding the first entry, sequence $Y$ repeats itself every $h_1h_2h_3$ entry.  Let $Y(\delta)=y_{\delta} y_{\delta+1}\ldots y_{\delta+h_1h_2h_3-1}$.  Observe that, when the extra 0 at the start of sequence $Y$ does not appear in this substring, $Y(\delta)$ must be a cyclic shift of $Y(1)=y_1 \ldots y_{h_1 h_2 h_3}$.  We therefore consider two cases: (i) those $\delta$ for which $y_0$ is not in $Y(\delta)$, and (ii) those $\delta$ such that $y_0$ is in $Y(\delta)$, namely $v-h_1h_2h_3+1 \leq \delta \leq v-1$.

In case (i), $Y(\delta)$  contains exactly  $h_2$ 1s which occur in a substring of length $h_1h_2$.
$X$ repeats itself every $h_1h_2$ entry within the first $h_1h_2h_3$ entries (which we are summing over). So the substring of $Y(\delta)$ of length $h_1h_2$ containing all the $h_2$ 1s must coincide with some cyclic shift of this.\\
The substrings we obtain are:

\[X\text{'s length-} h_1h_2 \text{ substring $X'$: } \overbrace{\overbrace{11\dots1}^{h_1}\overbrace{00\dots0}^{h_1}\dots\overbrace{00\dots0}^{h_1}}^{h_2\text{ blocks of length $h_1$}}\]

\[Y\text{'s length-} h_1h_2 \text{ substring  $Y'$: } \underbrace{\underbrace{0\dots01}_{h_1}\underbrace{0\dots01}_{h_1}\dots\underbrace{0\dots01}_{h_1}}_{h_2 \text{  repeats of length-$h_1$ blocks}}\]

We see that these binary sequences are precisely those corresponding to a non-disjoint ($h_1h_2, 2, h_1, h_2$)-GPSEDF $\{X_1,X_2\}$ of Theorem \ref{GPSEDFblock}.  Here $X'=X_1$ with $a_1 = h_1, \eta_1 = 1, b_1=h_2$, while $Y'$ is a cyclic shift of $X_2$ with $a_2=1, \eta_2 =1$ and $b_2=h_1$. The corresponding $a_0$ is $a_1b_1=h_1h_2$. By the proof of Theorem \ref{GPSEDFblock}, for any cyclic shift of $X_1$ and $X_2$ there are a constant number of 1s coinciding - this number is $\eta_1 a_1 \eta_2 a_2\frac{a_0}{a_1a_0}=\eta_1\eta_2a_2=1$. So our sum equals 1.

We now consider case (ii), which has two subcases. Although the additional $0$ at the start of $Y$ must now be taken into account, there is a range of values of $\delta$ for which the ``cyclic shift" behaviour remains unaffected.  When $v - h_1h_2h_3 +1 \leq \delta \leq v - (h_1(h_2-1))-1$, due to the position of the $0$s, $Y(\delta)$ is still a cyclic shift of $Y(1)$. To see this, note that $Y(\delta)$ still comprises (some shift of)
$h_2$ $1$'s each separated by gaps of $h_1-1$ and all other entries $0$, i.e. a cyclic shift of $y_1 \ldots y_{h_1 h_2 h_3}$. So we can use the same argument as above. 

Next, consider $\delta = v - (h_1(h_2-1))$.  Let $\delta'= v - (h_1(h_2-1))$.  Let $Y(\delta')=y_{\delta'} y_{\delta'+1} \ldots y_{\delta'+h_1 h_2 h_3-1}$; this is:
\[\underbrace{\underbrace{0\dots01}_{h_1}\underbrace{0\dots01}_{h_1}\dots\underbrace{0\dots01}_{h_1}}_{h_2-1 \text{ blocks of length } h_1}\textbf{0}\underbrace{0\dots0}_{h_1h_2(h_3-1)}\underbrace{0\dots0}_{h_1-1}
\]
These are the first $h_1 h_2 h_3$ terms of $Y$ after a shift by $\delta'$. Consider this shift of $Y$.  After the first length-$(h_1(h_2-1))$ substring, there is a gap of length $h_1 h_2 h_3-h_1 h_2+h_1$, i.e. the next leftmost position  containing symbol $1$ is position $h_1h_2h_3$.  Thus, for $Y(\delta)=y_{\delta} \ldots y_{\delta+h_1h_2h_3-1}$ with $v-(h_1(h_2-1))+1 \leq \delta \leq v-1$, i.e. $\delta=\delta'+r$ where $1 \leq r \leq h_1(h_2-1)-1$, the only $1$'s not arising from the initial $h_1(h_2-1)$ substring of $Y(\delta')$ (i.e. within the initial $h_1(h_2-1)-r$ entries of $Y(\delta)$) must occur in positions of $Y(\delta)$ greater than or equal to $h_1h_2h_3-r$ ($1 \leq r \leq h_1(h_2-1)-1)$, i.e. positions of $Y(\delta)$ between  $h_1h_2(h_3-1)+h_1+1$ and $h_1h_2h_3-1$. But $X$'s rightmost $1$ occurs in position $h_1h_2(h_3-1)+h_1-1$. This means that for $v-(h_1(h_2-1)) \leq \delta \leq v-1$, no contribution will be made to our sum from these positions. Hence we need only sum over the first $(h_2-1)h_1$ entries of $X$ and $Y(\delta)$ for $\delta$ in the range $v-h_1(h_2-1) \leq \delta \leq v-1$.

Note that $x_t=1$ for $0 \leq t \leq h_1-1$ and $0$ for $h_1 \leq t \leq h_1h_2-1$, so in fact we need only consider the first $h_1$ entries.  The initial length-$h_1(h_2-1)$ substring of $Y(\delta')$ repeats itself every $h_1$ entry. Since the first length-$h_1$-substring of this substring has weight 1, so does every other length-$h_1$-substring of this substring which lies within it. Finally the length-$h_1$ substrings of $Y(\delta')$ starting at positions to the right of $h_1 h_2-2h_1$  also have weight $1$ due to the presence of at least $h_1-1$ subsequent $0$s.
So our sum becomes:
$$
\sum_{t=0}^{v-1} x_ty_{t+\delta} = \sum_{t=0}^{(h_2-1)h_1-1} x_ty_{t+\delta}
=\sum_{t=0}^{h_1-1} y_{t + \delta}= 1$$

We have shown that this sum equals $1$ for all $\delta \in\{1,\dots,h_1h_2h_3h_4\}$. So $\Delta(A,B) = \mathbb{Z}_{h_1h_2h_3h_4+1} \setminus \{0\}$ and hence $\{A,B\}$ forms a classical ($h_1h_2h_3h_4,2,h_1h_3,h_2h_4,1,1$)-GSEDF in $\mathbb{Z}_{h_1h_2h_3h_4+1}$.
\end{proof} 

As one special case of this, we obtain Construction 3.8 of \cite{HucPatWEDF}

\begin{theorem}
Let $k_1, k_2 \in \mathbb{N}$, then $\{0, 1, \dots, k_1-1\}$ and $\{k_1, 2k_1, \dots, k_1k_2\}$ form a classical ($k_1k_2+1,2,k_1,k_2,1,1$)-GSEDF in $\mathbb{Z}_{k_1k_2+1}$.
\end{theorem}
\begin{proof}
We take $h_1=k_1, h_2=k_2, h_3=h_4=1$ in the above theorem.
\end{proof}

As another special case, we obtain Theorem 5.1 of \cite{HucJefNep}:

\begin{theorem}
Let $a \in \mathbb{N}$. Then $A=\{0,\dots,a-1\}\cup\{2a, 2a+1,\dots, 3a-1\}$ and $B =\cup_{i=1}^a \{(4i-1)a, 4ia\}=\{3a,4a,7a,8a,\dots,(4a-1)a,4a^2\}$ form a classical ($4a^2+1,2,2a,1$)-SEDF in $\mathbb{Z}_{4a^2+1}$.
\end{theorem}
\begin{proof}
We take $h_1=a$, $h_2=2$, $h_3=2$ and $h_4=a$ in the above theorem.
\end{proof}

We now state a general result for when the above Theorem gives a classical SEDF.

\begin{corollary}\label{SEDF}
Let $h_1,h_2,h_3,h_4 \in \mathbb{N}$, such that $h_1h_3=h_2h_4$. Then there exists a classical $((h_1h_3)^2 + 1, 2, h_1h_3, 1)$-SEDF in $\mathbb{Z}_{(h_1h_3)^2+1}$ with sets given by $A=\cup_{i=0}^{h_3-1} \{ih_1h_2 + \alpha : 0\leq \alpha \leq h_1-1\}$ and $B=\cup_{i=0}^{h_4-1} \{ih_1h_2h_3+ h_1h_2(h_3-1) + \beta h_1: 1 \leq \beta \leq h_2\}$.
\end{corollary}

We observe that the recursive construction in Theorem 5.6 (and hence Theorem 5.3) of \cite{HucJefNep} can also be proved using sequences; we omit the details.
\end{section}

\begin{section}{Applications in communications and cryptography}

We end by outlining some of the applications in communications and cryptography of the objects presented in this paper.

Non-disjoint PSEDFs have applications to optical orthogonal codes (OOCs)  \cite{HucNg}.   OOCs are sets of binary sequences used in optical multi-access communication \cite{ChuYan}.  For distinct binary sequences $X=(x_t)_{t=0}^{v-1}$ and $Y=(y_t)_{t=0}^{v-1}$, the autocorrelation of $X$ is a measure of how much it collides with shifts of itself ($\sum_{t=0}^{v-1} x_t x_{t+\delta}$, $1 \leq \delta \leq v-1$), while the cross-correlation of $X$ and $Y$ is a measure of how much $X$ collides with the shifts of $Y$ ($\sum_{t=0}^{v-1} x_t y_{t+\delta}$, $0 \leq \delta \leq v-1$).  Sequences $X$ and $Y$ are said to \emph{collide} in position $i$ if $x_i=1=y_i$.
The following definitions follow \cite{HucNg}.

\begin{definition}\label{def:CWOOC}
Let $v, w, \lambda_a, \lambda_c$ be non-negative integers with $v \ge 2$, $w \ge 1$.  Let ${\cal{C}}=\{X_0, \ldots, X_{N-1}\}$ be a family of $N$ binary sequences of length $v$ and weight $w$.  Then ${\cal{C}}$ is a (constant weight) $(v,w,\lambda_a, \lambda_c)$-OOC of size $N \ge 1$ if, writing $X=(x_i)_{i=0}^{v-1}$, $Y =(y_i)_{i=0}^{v-1}$ (indices modulo $v$): 
\begin{itemize}
\item[(i)] $\displaystyle{\sum_{t=0}^{v-1}} x_t x_{t+\delta} \le \lambda_a \mbox{ for any } X \in {\cal{C}}, 0 < \delta \le v-1$, and
\item[(ii)] $\displaystyle{\sum_{t=0}^{v-1}} x_t y_{t+\delta} \le \lambda_c \mbox{ for any } X, Y \in {\cal{C}}, 0 \le \delta \le v-1$.
\end{itemize}
\end{definition}
In other words, ${\cal{C}}$ is a $(v,w,\lambda_a, \lambda_c)$-OOC if auto-correlation values are at most $\lambda_a$ and cross-correlation values are at most $\lambda_c$. 

Via Definition \ref{sequenceset}, OOCs can equivalently be expressed as subsets of $\mathbb{Z}_v$ (see \cite{HucNg}).  For each sequence $X_i$, let $A_i$ be the set of integers modulo $v$ denoting the positions of the non-zero bits. The auto-correlation condition becomes
$$| A_i \cap (A_i + \delta) | \le \lambda_a \mbox{ for all $\delta \in
  \mathbb{Z}_v\setminus \{0\}$}$$
i.e. any non-zero $\delta$ occurs in $\Delta(A_i)$ at most $\lambda_a$   times, while the cross-correlation condition becomes
$$| A_i \cap (A_j + \delta) | \le \lambda_c
\mbox{ for all $\delta \in   \mathbb{Z}_v$}$$ i.e. any  $\delta$ occurs in $\Delta(A_i, A_j)$   most $\lambda_c$ times.

For applications, OOCs are required with low cross-correlation; in \cite{YanFuj} it is stated that the ``performance figure of merit" is the ratio between weight and cross-correlation.   In \cite{HucNg}, the following was proved:
\begin{proposition}\label{prop:optimalOOC}
If ${\cal{C}} = \{ X_0, \ldots, X_{N-1}\}$ is a $(v,w, \lambda_a, \lambda_c)$-OOC with $|{\cal{C}}| \ge 2$, then
$\lambda_{c} \ge \frac{w^2}{v}$.  Moreover, a $(v,w, \lambda_a, \lambda_c)$-OOC meeting the bound of (ii) with equality is a non-disjoint $(v,N,w,\frac{w^2}{v})$-PSEDF in $\mathbb{Z}_v$. 
\end{proposition}

All non-disjoint PSEDFs from Theorem \ref {generalPSEDF} yield OOCs which satisfy the conditions of Proposition \ref{prop:optimalOOC} and hence have lowest-possible cross-correlation.

In Definition \ref{def:CWOOC}, all sequence weights (and hence all corresponding set-sizes in $\mathbb{Z}_v$) are equal.  There is also a variable weight version of the OOC, called a VW-OOC, in which sequence weights are permitted to vary.  By the natural sequence/subset correspondence of Definition \ref{sequenceset}, non-disjoint GPSEDFs in $\mathbb{Z}_v$ yield examples of VW-OOCs (though we do not have an analogue of Proposition \ref{prop:optimalOOC}). The following definition follows that of \cite{ChuYan}.

\begin{definition}
Let $\mathcal{C} = \{C_a: 0 \leq a \leq N-1\}$ be a set of $N$ binary sequences
of length $v$, where $C_a=(c_a(t))_{t=0}^{v-1}$.  Then $\mathcal{C}$ is a $(v,W,\Lambda, \lambda_c, R)$-VW-OOC with $W = \{w_1,\ldots,w_m\}$,
$\Lambda = \{ \lambda_a(1), \ldots, \lambda_a(m)\}$, and $R = \{r_1,\ldots, r_m\}$ if the
following three conditions are satisfied:
\begin{itemize}
\item[(i)] there are exactly $r_i N$ codewords with weight $w_i$ for
$1 \leq i \leq m$, where $r_1 + \cdots + r_m = 1$ and $r_j \geq 0$ for
$1 \leq j \leq m$;
\item[(ii)] the autocorrelation of $C_a$ with weight $w_i$ in $C$ is upper bounded by some positive integer $\lambda_a(i)$; 
\item[(iii)] the cross-correlation between $C_a$ and $C_b$ with $a \neq b$ in
$C$ is upper bounded by some positive integer $\lambda_c$.
\end{itemize}
\end{definition}

All constructions of non-disjoint GPSEDFs in $\mathbb{Z}_v$ yield corresponding VW-OOCs.  Moreover, our approaches can be adapted to give new constructions for VW-OOCs which do not have the GPSEDF property.  We present an example of such a construction, in which the VW-OOCs have lowest-possible $\lambda_c=1$. (As with standard OOCs, low cross-correlation is a priority for VW-OOCs).   These are distinct from known constructions in \cite{ChuYan} and \cite{Yan}.

\begin{theorem}\label{thm:VW}
Let $a_1,\ldots,a_m$ be distinct divisors of $v$ such that for each $i \neq j$, $v=lcm(a_i,a_j)$.  For $1 \leq i \leq m$, define the binary sequence $X_i=(x_t)_{t=0}^{v-1}$ where:
    \[x_t = \begin{cases}
	1,& t=la_i\quad  (0 \leq l \leq \frac{v}{a_i}-1)\\
	0,& \text{otherwise}
	\end{cases}\]
Then $\{X_1,\ldots, X_m\}$ form a VW-OOC in which $N=m$, $r_i=\frac{1}{m}$, $w_i=\frac{v}{a_i}$, $\lambda_a(i)=\frac{v}{a_i}$ and $\lambda_c=1$.
\end{theorem}
\begin{proof}
By construction, there is one sequence $X_i$ of each distinct weight $w_i=\frac{v}{a_i}$, so $r_i=\frac{1}{m}$.  Observe that the subset $A_i$ of $\mathbb{Z}_v$ corresponding to the sequence $X_i$ is the subgroup $A_i= a_i \mathbb{Z}_v$.  The internal difference multiset $\Delta(A_i)$ comprises each element of $A_i$ precisely $|A_i|$ times and so also $\lambda_a(i)=\frac{v}{a_i}$.  By Lemma \ref{lem:gcd}, for any pair $i \neq j$, $\Delta(A_i,A_j)$ comprises $\frac{v}{lcm(a_i,a_j)}=1$ copy of each element of $gcd(a_i,a_j) \mathbb{Z}_v \subseteq \mathbb{Z}_v$ and $0$ copies of $\mathbb{Z}_v \setminus gcd(a_i,a_j) \mathbb{Z}_v$, so $\lambda_c=1$.
\end{proof}

\begin{example}
Let $v=p_1 p_2 \ldots p_m$ where the $p_i$ are distinct primes ($1 \leq i \leq m$).  For $1 \leq i \leq m$, let $a_i=\frac{v}{p_i}$.  Clearly for $i \neq j$, $lcm(a_i,a_j)=v$, so we may apply the construction of Theorem \ref{thm:VW}.\\
Specifically, let $v=30=2 \cdot 3 \cdot 5$, where $a_1=6$, $a_2=10$ and $a_3=15$.  Then $N=3$ and we obtain a $(30, \{2,3,5\}, \{2,3,5\}, 1, \{\frac{1}{3},\frac{1}{3},\frac{1}{3}\})$ VW-OOC.  
\end{example}

Theorem \ref{thm:VW} can easily be extended, using Lemma \ref{lem:gcd}(ii) and the approach of Theorem \ref{thm:AlgGps}, to yield VW-OOCs with larger weights and larger $\lambda_c$, by taking suitable unions of cosets of subgroups of $\mathbb{Z}_v$.

In a different communications context, the sequence version of a non-disjoint GPSEDF in $\mathbb{Z}_v$ will satisfy the requirement for a set of periodic binary sequences (with common period $v$, or some divisor of $v$) to be \emph{pairwise shift invariant}  \cite{ZhaShuWon}.  A binary sequence $S=(x_i)$ (of unspecified length) is said to be periodic with period $L$ if $x_i=x_{i+L}$ for all $i \geq 1$ (with $L$ smallest possible). 

\begin{definition}[\cite{ZhaShuWon}]
Let $S_1,\ldots,S_k$ be binary sequences with a common period $L$; denote the $t$-th entry of $S_i$ by $S_i(t)$.  Their $k$-wise Hamming cross-correlation for relative shifts $\tau_1,\ldots,\tau_{k-1}$ is defined as
$$ H_{S_1,\ldots,S_k}(\tau_1,\ldots,\tau_{k-1})=\sum_{t=0}^{L-1}S_1(t) S_2(t+\tau_1) \ldots S_k(t+\tau_{k-1})$$
(the normalized version is $\hat{H}_{S_1,\ldots,S_k}=\frac{1}{L}H_{S_1,\ldots,S_k})$.
For $k=2$, this is the usual Hamming cross-correlation.
This is called shift invariant (SI) if $\hat{H}_{S_1,\ldots,S_k}$ is equal to a constant; a set of sequences is completely $SI$ if this is constant for all choices of $k$ sequences and all $k$; pairwise SI  if the 2-wise Hamming cross-correlation is SI for all pairs.
\end{definition}

In defining multiple access protocols, it can be desirable to have binary sequences which are completely SI; for some applications it is of interest to relax the condition to the pairwise case. Under certain conditions, a set of pairwise SI sequences in fact satisfies the stronger property of being completely SI:  

\begin{theorem}\label{thm:duty}[Theorem 3, \cite{ZhaShuWon}]
Let $p$ be a prime. 
For binary sequence $S$ with period $L$, its duty factor is defined to be $\frac{1}{L}\sum_{i=0}^{L-1} S(i)$. 
If $K$ pairwise SI protocol sequences with duty factors $\frac{n_i}{p}$, for $i=1,2,\ldots,K$ have a common minimum period $p^K$ , then they are completely SI.
\end{theorem}

All of our non-disjoint GPSEDFs in $\mathbb{Z}_v$ yield examples of pairwise SI sequences with common period dividing $v$, and some of our constructions (e.g. Corollary \ref{cor:powers}(i)) satisfy the condition of Theorem \ref{thm:duty} to be completely SI.  Conversely, the constructions of pairwise SI sequences with common period $L$ in \cite{ZhaShuWon} (defined via cyclotomic polynomials) give examples of non-disjoint GPSEDFs in $\mathbb{Z}_L$.

\begin{example}
Applying the non-disjoint GPSEDF construction of Corollary \ref{thm:3sets} in $\mathbb{Z}_{30}$ with $a_1=6, a_2=15$ and $a_3=10$ yields Example 1 of \cite{ZhaShuWon}.
\end{example}

Finally, the classical two-set SEDFs obtainable from Theorem \ref{GSEDF} and Corollary \ref{SEDF} have applications to classical circular external difference families (CEDFs) and circular AMD codes \cite{VeiSti, WuYanFen, PatSti2}, in particular the strong version of these (SCEDFs). These are recently-introduced structures used in cryptography to construct non-malleable threshold schemes. Classical two-set SEDFs give examples of classical two-set SCEDFs; it has been shown that no more than two sets are possible in classical SCEDF \cite{WuYanFen}.  We note that, if the non-disjoint version of the SCEDF is defined in the natural way, then any non-disjoint PSEDF yields a non-disjoint SCEDF, and hence non-disjoint SCEDFs exist with any number of sets.
\end{section}

\begin{section}{Conclusion}
We end with a table summarising the constructive and structural results presented in this paper. We hope that these results will stimulate further research in the area.
    
In the table, we use the following abbreviations; Thm: Theorem, Prop: Proposition, Cor: Corollary and Lem: Lemma.\\

    \begin{center}
    \begin{tabular}{| c| c | c |}
    \hline
    & Constructions & Structural Results \\ 
    \hline
    Classical SEDF & Thm 6.3, Cor 6.4 & \\
    \hline
    Classical PSEDF & & Thm 2.22 \\ 
    \hline
    Classical GSEDF & Thm 6.1, Thm 6.2 & Thm 2.23(ii) \\   
    \hline
    Classical GEDF & Thm 4.12(ii) & \\
    \hline
    \end{tabular}
    \\
    Table 1: A list of the classical results included in the paper
    \\

    \begin{tabular}{| c| c | c |}
    \hline
    & Constructions & Structural Results \\
    \hline
    Non-Disjoint PSEDF & Thm 3.4, Cor 3.5,& Prop 2.7(iii) \\ 
    & Cor 3.6, Cor 4.4(i), Thm 5.1 & \\
    \hline
    Non-Disjoint GPSEDF & Prop 2.12(ii), Prop 2.13(v), Thm 2.25, & Prop 2.13(v),\\
    & Thm 3.2, Cor 3.3, Thm 3.7, & Thm 2.24 \\
    & Cor 3.8, Thm 3.10, Thm 4.3, & \\
    & Cor 4.4, Lem 4.7, Thm 4.9,& \\
    &  Cor 4.10, Thm 4.12(i), Cor 5.3(i) &\\  
    \hline
    Non-disjoint GSEDF & Prop 2.12(ii), Cor 5.3(iii) & Thm 2.23(i) \\  
    \hline
    Non-disjoint MGPSEDF & Thm 3.12, Cor 5.3(ii) & Thm 2.24(iv,v) \\  
    \hline
    Non-disjoint MGSEDF & Cor 5.3(iv) & \\
    \hline
    \end{tabular}
    Table 2: A list of the non-disjoint results included in the paper
    \end{center}
    
\end{section}
\begin{subsection}*{Acknowledgements}
The first author was funded by EPSRC grant EP/X021157/1.  The second author was funded by the London Mathematical Society grant URB-2023-75.
\end{subsection}

\end{document}